\documentclass[11pt,a4paper]{article}
\usepackage[utf8]{inputenc}
 \usepackage{a4wide} 
\usepackage[utf8]{inputenc}
\usepackage{subfig}
\usepackage{bbm}
\usepackage{amsmath}
\usepackage{graphicx}
\usepackage{epstopdf}
\usepackage{algcompatible}
\usepackage{adjustbox}
\usepackage{titlesec}
\usepackage{csquotes}
\RequirePackage{ifthen}
\RequirePackage[T1]{fontenc}
\RequirePackage[ngerman,english]{babel}
\RequirePackage{amsmath}
\RequirePackage{amssymb}
\RequirePackage{amsthm}
\RequirePackage{amsfonts}
\RequirePackage{graphicx}
\RequirePackage{color}
\RequirePackage{listings} 
\usepackage{epstopdf}

\usepackage{algcompatible}
\RequirePackage[section,Algorithm]{algorithm}
\RequirePackage{emptypage}

\usepackage[backend=biber,sorting=nyt]{biblatex}
\renewbibmacro{in:}{}
\DeclareFieldFormat[article]{title}{\textit{#1}} 
\DeclareFieldFormat[article]{journaltitle}{#1}
\theoremstyle{plain}						
\newtheorem{theorem}{Theorem}[section]
\newtheorem{lemma}[theorem]{Lemma}

\numberwithin{equation}{section}

\titleformat{\section}[block]{\normalfont\bfseries}{\thesection.}{0.5em}{}
\titlespacing{\section}{0pc}{1pc}{1pc}

\titleformat{\subsection}[block]{\normalfont\bfseries}{\thesubsection.}{0.5em}{}
\titlespacing{\subsection}{0pc}{1pc}{1pc}

\begin{document}
 
\title{Microscopic and Macroscopic Traffic Flow Models\\ including Random Accidents}

\author{Simone G\"ottlich\footnotemark[1], \; Thomas Schillinger\footnotemark[1]}

\footnotetext[1]{University of Mannheim, Department of Mathematics, 68131 Mannheim, Germany (goettlich@uni-mannheim.de, schillinger@uni-mannheim.de)}

\date{ \today }
\maketitle

\begin{abstract}
\noindent
We introduce microscopic and macroscopic stochastic traffic models including traffic accidents. The microscopic model is based on a Follow-the-Leader approach whereas the macroscopic model is described by a scalar conservation law with space dependent flux function. Accidents are introduced as interruptions of a deterministic evolution and are directly linked to the traffic situation. Based on a Lax-Friedrichs discretization convergence of the microscopic model to the macroscopic model is shown. Numerical simulations are presented to compare the above models and show their convergence behaviour.
\end{abstract}




\section{Introduction}

Throughout the world, traffic accidents are a serious problem and causes considerable societal costs.  
So there is a great interest in understanding how accidents may happen and how they may be reduced. Mathematical models can help at least to analyze traffic scenarios and probably allow for a reliable prediction.  
In particular, there exist a variety of different mathematical approaches to model traffic accidents, for instance ordinary differential equations \cite{accidents6}, kinetic models \cite{kinMod}, conservation laws \cite{,accidents1.1,accidents1}, 
queuing theory \cite{accidents2}, 
and statistical methods \cite{accidents3,accidents5,accidents4}. 
Recently, some works have been published dealing with the question how to reduce the risk of an accident. Especially, by introducing autonomous vehicles it is suggested to reduce the variation in the vehicles velocities to decrease the likeliness of an accident \cite{add2,add1}.

For the detailed description of the traffic dynamics, we distinguish between two scales, i.e. microscopic and macroscopic. In a microscopic setting each vehicle is considered as a particle that moves on a road and adjusts its velocity according to the behaviour of other vehicles in front. Such models have been introduced by Pipes \cite{Pipes} in 1953. The dynamics are governed by ordinary differential equations (ODE) for each vehicle depending on the distance of the vehicle in front and therefore called Follow-the-Leader models \cite{FtL2,FtL1,FollowLeader,FtL3}. On the other hand, a macroscopic approach considers traffic as a density that floats on the road. These kind of traffic models were introduced by Lighthill-Whitham-Richards (LWR) in \cite{LWR,LWR2} and have been further extended e.g. in \cite{piccoli1, piccoli2,stochmac2, stochmac1}. 
Several works also covered the problem of convergence of the microscopic to the macroscopic model, for instance \cite{MicMac4, MicMac5, MicMac1, MicMac3, MicMac2, FollowLeader}. 

In this paper, we introduce two models based on deterministic traffic dynamics which face accidents as a stochastic disturbance. Similar ideas of a deterministic system that is interrupted by random events have been considered in production networks, e.g. \cite{RandomMachineBreakdowns3, RandomMachineBreakdowns1}. They have been also extended to a model with production dependent breakdowns \cite{15}. A similar model for pedestrians was developed in \cite{pedestrians} where individuals switch randomly between stop and go. Here, we adapt the idea of process dependent interruptions to a hierarchy of traffic models and investigate their relation applying a convergence analysis. Specifically, we focus on two types of accidents: First, we consider accidents due to high traffic flux which correspond to the idea that accidents are more likely having both, a high density and high velocity of the vehicle. Second, we investigate rear-end collisions which for example can be observed at tailbacks.

In contrast to other accident models, we consider a framework in which traffic and accidents are connected in a bi-directional relation. More precisely, for the microscopic model we add traffic accidents by adjusting the velocity function of the piecewise deterministic ODE-system in the area where an accident happens (characterized by an accident position and an accident size). On the other hand, the probability measures governing the accident evolution are determined by the vehicle positions. The macroscopic model under consideration was developed in \cite{A} and presents a deterministic dynamic that is interrupted by random accidents. The time and the position of an accident are modeled dependent on the traffic situations that are observed. An accident then influences the flux function and reduces the capacity of the system in the area of the accident. So the flux function may be space dependent, see for example \cite{B,spaceDep1,spaceDep2}. We also study the convergence in this framework. Concentrating on a microscopic submodel in which the stochastic influences are governed by the macroscopic model, we use Lagrangian variables and show their convergence to a related conservation law based on a Lax-Friedrichs discretization. An equivalence theorem then provides a limit for the conservation law whose weak solution is just the macroscopic traffic density in Eulerian variables.

This paper is organized as follows: In Section $\ref{sec:BasicModel}$ we introduce a microscopic traffic accident model based on a Follow-the-Leader approach and provide the conditional transition probabilities for a stochastic process. Afterwards in Section $\ref{sec:MakroModelll}$, we present the corresponding LWR type macroscopic traffic model. Using the macroscopic stochastic components regarding the accidents for a microscopic model we define a third model and investigate this micro-macro-limit in the Sections $\ref{TildeModel}$ and $\ref{sec:ConvergenceMikroMakro}$. In Section $\ref{numRes}$, we provide the numerical treatment for both, the microscopic and macroscopic model, and additionally present a numerical convergence analysis using different error measures.

\section{Traffic Accident Models}
In this section we present the microscopic and macroscopic traffic accident models and describe how accidents are can be incorporated.

\subsection{Microscopic Model}\label{sec:BasicModel}
We introduce a standard deterministic Follow-the-Leader microscopic traffic model (see e.g. \cite{FollowLeader}) and consider one lane of a one-way-street with a total number of $N \in \mathbb{N}$ cars. The road is modelled by an interval $[a,b] ,~ a<b$ on which we assume periodic boundary conditions. Let $x_i(t)$ denote the position of vehicle $i \in \lbrace 1,...,N \rbrace$ at time $t>0$. The development of the position of any vehicle with respect to the time is given by the following system of ordinary differential equations for $i=1,...,N-1$
\begin{align}
\begin{split}
\label{det_ode}
\dot{x}_i (t) &= v \bigg(\frac{L}{\Delta x_i (t)}\bigg),~~~ \dot{x}_N(t) = v \bigg(\frac{L}{x_1 - x_N + (b-a)}\bigg).
\end{split}
\end{align}
The parameter $L>0$ denotes the length of the cars and $\Delta x_i(t) = x_{i+1}(t) - x_i(t)$ the distance between the fronts of vehicle $i$ and $i+1$ at time $t$. Note that $\Delta x_i(t)$ is always at least $L$, such that the argument in the velocity function in $(\ref{det_ode})$ is smaller or equal to 1. We assume $v: \mathbb{R} \rightarrow \mathbb{R}$ to be a Lipschitz continuous function modeling the velocity. In the case of a Follow-the-Leader model a reasonable choice is $v(\rho) = 1- \rho$ for $\rho \in [0,1]$. Cars are not allowed to overtake each other such that they move in a fixed order. 

To adapt to different road capacities one may add a road capacity function $c_{road}: \mathbb{R} \rightarrow \mathbb{R}$ depending on the position of a vehicle. In our particular case this function will be uniformly bounded, piecewise constant and have at most finitely many jumps. Using this coefficient we may incorporate different speed limits at different locations on the road. Then for $i=1,...,N-1$ the equation $(\ref{det_ode})$ expands to
\begin{align}
\begin{split}
\label{odec1}
\dot{x}_i (t) &= c_{road}(x_i(t)) v \bigg(\frac{L}{\Delta x_i(t)}\bigg), ~~  \dot{x}_N(t) = c_{road}(x_N(t)) v \bigg(\frac{L}{x_1 - x_N + (b-a)}\bigg).
\end{split}
\end{align}

In the model, accidents consist of three components: The location, the size and the capacity reduction of an accident. The location represents the exact position where the accident happens and then starts to affect the system. The size is the parameter that models  to which length it affects the system of vehicles. Combining the accident location and the accident size we obtain an interval on which the accident has some influence. The extent of the accident, now considered in a fraction of how much of the road capacity is being reduced by the accident, is given in the last component, the capacity reduction. 

If there is more than only one accident, the leftover road capacities are multiplied to obtain the resulting road capacity. In the model an accident has the same impact on the road as an additional speed limit which also reduces the number of vehicles passing a certain position on the road.

To incorporate the accidents in the system of vehicle positions, we introduce an accident capacity function. Assume $M \in \mathbb{N}_0$ to be the number of accidents being active. We equip each accident with an index $j \in \lbrace 1,...,M \rbrace$ and use the vectors $p \in \mathbb{R}^M$ for the positions, $s \in \mathbb{R}_+^M$ for the sizes and $c \in [0,c_{max}]^M$ for the capacity reductions of each accident, for $c_{max} \in [0,1)$.  Then the accident capacity function is given by
\begin{equation*}
c_{ac}: \mathbb{R} \rightarrow \mathbb{R}, ~~ x \mapsto \prod_{j=1}^{M} \Big(1-c_j \mathbbm{1}_{\big[p_j - \frac{s_j}{2}, p_j + \frac{s_j}{2}\big]} (x)\Big).
\end{equation*}
The index of the accident capacity function is meant to contain all the necessary information about the current accident situation.
Then the system $(\ref{odec1})$ extends to
\begin{align}
\begin{split}
\label{odeAcc}
\dot{x}_i (t) &= c_{road}(x_i(t))c_{ac}(x_i(t)) v \bigg(\frac{L}{\Delta x_i}\bigg), ~~ i=1,...,N-1  \\
\dot{x}_N(t) &= c_{road}(x_N(t))c_{ac}(x_N(t)) v \bigg(\frac{L}{x_1 - x_N + (b-a)}\bigg).
\end{split}
\end{align}
To obtain a well defined model, we assume that for the initial vehicle positions it holds 
\begin{align}
\begin{split}
\label{ICMikro}
L &< |x_{i+1}(0) - x_i(0) |, ~~ L < |x_{1}(0) - x_N(0) + b-a | \\
a &\leq x_1 < ... < x_N < b , \text{   for } i=1,...,N-1.
\end{split}
\end{align}
This condition ensures that the distance between any two cars at the initial state is bounded from below by the vehicle length and that the cars are ordered according to their indices in a road segment with finite length.
 
In this model we do not actually want to have a vehicle crashing into the other, in terms that their position overlap. To ensure that our model is well defined in that sense and does not generate collisions, we require the following condition on the time grid $(t_j)_{j \in \mathbb{N}}$: Let for any vehicle $i=1,...,N-1$, and any time with $t_{j+1} = t_j + \Delta t$ hold true
\begin{equation*}
x_{i+1}(t_{j+1}) - x_i(t_{j+1}) - L \geq 0, ~~~ x_{1}(t_{j+1}) - x_N(t_{j+1}) - L + (b-a) \geq 0
\end{equation*}
assuming that $x_{i+1}(t_{j}) - x_i(t_{j}) - L \geq 0$. This condition makes sure that the difference between the front of two vehicles is at least the vehicle length, such that the rear vehicle does not cause a collision. The inequality can be achieved by setting the step size $\Delta t$ small enough. We use an Euler approximation to derive a condition for $\Delta t$:
\begin{align*}
x_{i+1}(t_{j+1}) - x_i(t_{j+1}) - L =&~ x_{i+1}(t_j) + \Delta t c(x_{i+1}(t_j))\bigg(1 - \frac{L}{x_{i+2}(t_j) - x_{i+1}(t_j)}\bigg)\\
&- \bigg(x_{i}(t_j) + \Delta t c(x_{i}(t_j)) \bigg(1 - \frac{L}{x_{i+1}(t_j) - x_{i}(t_j)}\bigg)\bigg) - L \\
\geq&~ x_{i+1}(t_j) - x_i(t_j) - \Delta t v_{max} \bigg(1 - \frac{L}{x_{i+1}(t_j) - x_{i}(t_j)}\bigg) - L \\
 \geq&~ \underbrace{(x_{i+1}(t_j) - x_{i}(t_j) - L)}_{\substack{\geq 0}} \bigg(1 - \frac{\Delta t v_{max}}{L} \bigg).
\end{align*}
The first factor is non-negative per assumption. Thus, the whole product remains positive if
\begin{equation}
\label{diskreteCFL}
\Big(1 - \frac{\Delta t v_{max}}{L} \Big) \geq 0~ \Leftrightarrow ~ \Delta t \leq \frac{L}{v_{max}}.
\end{equation}
Nevertheless, our model will capture accidents, but they are rather generated artificially by some probability distributions that model the likeliness of an accident. 
We assume that there are basically two types of accidents. The first type consists of accidents due to both, high traffic density and high velocity of the vehicles. For the second type one may also think of a higher likeliness of accidents when a driver faces an increase of the traffic density (e.g. at the end of a traffic jam). Such a rear-end collision at the end of a traffic jam often is caused by drivers misjudging the distance and speed of the vehicles in front of them. 

To model the first type of accidents we note that in general at a higher speed of the cars an accident is more likely. Additionally, a higher density of cars leads to a higher likelihood of an accident. Thus, it is reasonable to use the product of both to define a measure for the probability of an accident given a time and a range in space. This choice also reflects that if we have only one vehicle on our road (density of 0) and even the speed is $v_{max}$, we do not expect any accidents of this type. The same effect applies to the situation where we have a maximum density, and we face a bumper-to-bumper situation (i.e. cars have no buffer between them). Then, the velocity is zero and thus an accident will not be possible. 

To be more precise about the measure for the accident location we introduce some definitions.
For each vehicle $i=1,...,N-1$ we define a local density by
\begin{equation*}
\label{locdensity}
 \rho_i (t) = \frac{L}{x_{i+1}(t) - x_i(t)}, ~~~ \rho_N (t) = \frac{L}{x_{1}(t) - x_N(t) + (b-a)}.
\end{equation*}
Taking into account periodic boundary conditions, we define a piecewise constant function $h_{ac}: \mathbb{R} \times \mathbb{R}_+ \rightarrow \mathbb{R}$ for $i=1,...,N-1$ by
\begin{equation*}
h_{ac}(x,t) =      
\begin{cases}
    c_{road}(x_i(t))c_{ac}(x_i(t)) \rho_i (t) v(\rho_i(t)) &,~ x_i(t)\leq x < x_{i+1}(t) \\
    c_{road}(x_N(t))c_{ac}(x_N(t)) \rho_N (t) v(\rho_N(t)) &,~ x \in [x_N(t), b] \cup [a,x_1(t)) \\
	0 &    \text{, else.}
    \end{cases} 
\end{equation*}
A time and accident situation dependent constant is defined by 
\begin{align}
\begin{split}
\label{C_F_const}
C_{F_{ac}} (t) =& \int_\mathbb{R} h_{ac}(x,t) dx \\
=& \sum_{i=1}^{N-1} c_{road}(x_i(t))c_{ac}(x_i(t)) \rho_i (t) v(\rho_i(t)) (x_{i+1}(t) - x_{i}(t))\\
 &+ c_{road}(x_N(t))c_{ac}(x_N(t)) \rho_N (t) v(\rho_N(t)) (x_{1}(t) - x_{N}(t)+(b-a)).
\end{split}
\end{align}
Here, $C_{F_{ac}} (t)$ is finite if $(\ref{ICMikro})$ holds because the road capacities, local densities and velocities are finite and the sum of the vehicle distances can at most amount to the length of the road.

Using $(\ref{C_F_const})$, we construct a family of probability measures $(\mu_{t}^{F_{ac}})_{t \geq 0}$ on $(\mathbb{R}, \mathcal{B}(\mathbb{R}))$ for the position of a \textbf{accident of type 1} at time $t$ and a given accident situation by
\begin{equation}
\label{muFMik}
\mu_{t}^{F_{ac}} (B) = \frac{\int_B h_{ac}(x,t) dx}{C_{F_{ac}}(t)}
\end{equation}
for $B \in \mathcal{B}(\mathbb{R})$, the Borel $\sigma$-algebra over $\mathbb{R}$.

A probability measure can also be defined for \textbf{accident type 2}. This should be reflecting situations where accidents occur at the tail of a traffic jam or any other increase in the traffic density. Mathematically, this can be expressed by considering the difference of the local densities of two consecutive vehicles. If such a difference is positive we observe an increase in the local densities which corresponds to denser traffic in front of a vehicle. Denser traffic also corresponds to vehicles with lower speed and thus evokes the risk of a rear-end collision.
Therefore, we set
\begin{equation}
\label{Drplus}
D \rho_{+} (t) = \sum_{i=1}^{N-1} (\rho_{i+1} (t) - \rho_i (t))_{+} + (\rho_{1} (t) - \rho_N (t))_{+},
\end{equation}
where $(x)_{+} = \max \lbrace x,0 \rbrace$. 
A family of probability measures $(\mu_t^D)_{t\geq 0}$ on  $(\mathbb{R},\mathcal{B}(\mathbb{R}))$ for the position of an accident type 2 can then be constructed by
\begin{equation}
\mu_t^D (B) = \frac{\sum_{i=1}^{N-1} \varepsilon_{x_i(t)} (B) (\rho_{i+1} (t) - \rho_i (t))_{+} + \varepsilon_{x_N(t)} (B) (\rho_{1} (t) - \rho_N (t))_{+}}{D \rho_{+} (t)}
\end{equation}
for $B \in \mathcal{B}(\mathbb{R})$ and $\varepsilon_x(B)$ being the Dirac measure in $x$ for $B$. 

To combine both accident types, we choose a parameter $\beta \in [0,1]$ which describes the share of accidents of type 1. Accordingly, accidents of type 2 occur with share $(1-\beta )$. The family of probability measures $(\mu_{t,ac}^{pos})_{t\geq 0}$ on $(\mathbb{R}, \mathcal{B}(\mathbb{R}))$ reflecting the position of an accident is given by
\begin{equation}
\label{mupos}
\mu_{t,ac}^{pos} (B) = \beta \mu_{t}^{F_{ac}} (B) + (1- \beta) \mu_t^D (B)
\end{equation}
for $B \in \mathcal{B}(\mathbb{R})$.

As the accident avoidance plays an important role, we remark that for accidents of type 2, the accident risk can be decreased if the distances between the individual vehicles are of similar length. This reduces large oscillations in the local density functions and therefore decreases $D\rho_+$. Such an effect can be achieved by aligning the vehicle velocities. On the other hand accidents of type 1 can be considered as a background accident noise which is not directly influenced by certain driver behaviours. Therefore, an accident prevention strategy is not that obvious for this type of accidents.

So far only the probability distribution for the position of an accident has been investigated. To model an accident properly, we also assign a size and a road capacity reduction. We introduce the probability measures $\mu^{s}$ and $\mu^{cap}$ on $(\mathbb{R}, \mathcal{B}(\mathbb{R}))$ for the size and the capacity reduction of the accident respectively. For instance for both measures one could choose a uniform distribution on some appropriate interval. 

In a next step we define a time discrete stochastic process that models the whole accident traffic model. Therefore, let us now consider an equidistant time grid $({t_n})_{n \in \mathbb{N}}$ with step size $\Delta t >0$. We introduce a stochastic process $X = (X_n, n \in \mathbb{N})$ in a state space $E^{(N,K)}$, where 
\begin{align*}
X_n &= (x^n, M^n, p^n, s^n, c^n, u^n, l^n) \in E^{(N,K)}, \\
E^{(N,K)}&= \mathbb{R}^{N} \times \mathbb{N}_0 \times \mathbb{R}^{K} \times \mathbb{R}^{K} \times \mathbb{R}^{K} \times \lbrace -1,0,1 \rbrace \times \mathbb{N}
\end{align*}
on some probability space $(\Omega, \mathcal{A}, P)$. Set $K \in \mathbb{N}$ large enough to capture all accidents that may happen. As before, $N \in \mathbb{N}$ determines the total number of vehicles.

We denote for any time $t_n$, $x_i^n$ the position of vehicle $i$, $M^n$ the total number of accidents at $t_n$, $p_j^n$ the position of accident $j$, $s_j^n$ the size of accident $j$ and $c_j^n$ the capacity reduction of accident $j$. Let $u^n$ be the component that indicates the kind of event that happens. We assign $u^n=-1$ in the case that an accident dissolves, whereas $u^n=1$ indicates when a new accident is going to occur and $u^n=0$ in the case neither of these happen in $t_n$. The component $l^n$ is an additional parameter for the accident index that is going to dissolve, in case an accident dissolves at $t_n$.

To determine when and which event occurs, we use rate functions that describe the likeliness of an event given the current traffic situation. The ideas for these rate function are based on concepts introduced in \cite{A}. We choose an appropriate rate function $\psi : E \times \mathbb{R}_+ \rightarrow \mathbb{R}_+$.
The rate depends on two given parameters $\lambda^F>0$ and $\lambda^D>0$ for the likeliness of an accident due to high flux or ends of tailbacks, respectively. They both may depend on $\beta$, i.e. $\lambda^F = \beta k_1$ and $\lambda^D = (1- \beta) k_2$ with two constants $k_1, k_2 \in \mathbb{R}_+$. 

Additionally, the rate function takes the two constants $C_{F_{ac}}$ defined as in $(\ref{C_F_const})$ and $D \rho_+$ defined as in $(\ref{Drplus})$ into account. The larger $C_{F_{ac}}$ the higher the risk of an accident of type 1. The larger $D \rho_+$ the more increases in the local density functions can be observed (which more or less correspond to ends of traffic jams) and thus increases the general risk of an accident of type 2. The rate function for a new accident is given by the function $\lambda^A: E \times \mathbb{R}_+ \rightarrow \mathbb{R}$ and
\begin{equation*}
\label{ratefct}
\lambda^A (y,t) = \lambda^F  C_{F_{ac}}(t) + \lambda^D  D \rho_+(t).
\end{equation*} 
The accident situation is now incorporated into $C_{F_{ac}}$ via information out of the stochastic process. 
We extend $\lambda^A$ by introducing a rate $\lambda^R >0$ for the dissolution of an accident weighted by the number of accidents in
\begin{equation}
\label{ratefct2}
\psi (y,t) = \lambda^F C_{F_{ac}}(t) + \lambda^D  D \rho_+(t) + \lambda^R M(t). 
\end{equation}
This enables us to define the transition probabilities of the stochastic process. Let $\Delta t >0$ satisfy $(\ref{diskreteCFL})$. 
The probabilities for the family $(u^n)_{n \in \mathbb{N}}$ describing the kind of event which is going to happen is given by $(\ref{probu})$ - $(\ref{lastprob})$ using the rate function introduced in $(\ref{ratefct2})$
\begin{alignat}{2}
&P(u^{n+1} = 0 \mid X^n) &&= 1 - \Delta t \psi(X^n) \label{probu}\\
&P(u^{n+1} = 1 \mid X^n) &&=  \Delta t \psi(X^n) \frac{\lambda^A}{\lambda^A + M^n \lambda^R} \\
&P(u^{n+1} = -1 \mid X^n) &&=  \Delta t \psi(X^n) \frac{M^n \lambda^R}{\lambda^A + M^n \lambda^R}. \label{lastprob}
\end{alignat}
Using the family of parameters $(u^n)_{n \in \mathbb{N}}$ we can directly specify the conditioned transition probabilities for the number of accidents in a next time step
\begin{alignat}{2}
&P(M^{n+1} = M^n \mid X^n) &&= \mathbbm{1}_0(u^n)\\
&P(M^{n+1} = M^n+1 \mid X^n) &&= \mathbbm{1}_1(u^n)\\
&P(M^{n+1} = M^n-1 \mid X^n) &&= \mathbbm{1}_{-1}(u^n).
\end{alignat}
After having chosen $u^n$ there obviously is no uncertainty in the development of the evolution of the number of accidents.
Another family of random variables $(l^n)_{n \in \mathbb{N}}$ is used to determine which accident is going to be dissolved at time $t_n$ if this case occurs. As we want to ensure that there is no underlying pattern in the dissolving accidents we uniformly pick an index of the currently active accidents in $(\ref{probL})$ 
\begin{align}
&P(l^{n+1} = k \mid X^n) = \mathbbm{1}_{[0,M^n]}(k) \frac{1}{M^n}.    \label{probL} 
\end{align}
The equations $(\ref{x_1})$ and $(\ref{x_2})$ correspond to Euler approximations of the system of ordinary differential equations in $(\ref{odeAcc})$ and model the position of each vehicle 
\begin{alignat}{1}
&P\Big(x_i^{n+1} = x_i^n + \Delta t c_{road}(x_i^n)  c_{ac}(x_i^n)  v \bigg( \frac{L}{\Delta x_i^n} \bigg)\mid X^n\Big) = 1,~~ i=1,...,N-1 \label{x_1} \\
&P\Big(x_N^{n+1} = x_N^n + \Delta t c_{road}(x_N^n)  c_{ac}(x_N^n)  v \bigg( \frac{L(N)}{x_{1}^n-x_N^n + (b-a)} \bigg)\mid X^n \Big) = 1.\label{x_2}
\end{alignat}
As a next step we consider the transition probabilities of the three accident parameters
\begin{alignat}{2}
&P(p_k^{n+1} = p_k^n \mid X^n) &&= \mathbbm{1}_0(u^{n}) + \mathbbm{1}_1(u^{n}) (1-\mathbbm{1}_{M^n +1}(k)) +\mathbbm{1}_{-1} (u^{n}) (1 - \mathbbm{1}_{l^n}(k)) \label{p_1}\\
&P(p_k^{n+1} = 0 \mid X^n) &&= \mathbbm{1}_{-1}(u^{n}) \mathbbm{1}_{l^n}(k)  \label{p_2} \\
&P(p_k^{n+1} \in B_1 \mid X^n) &&= \mathbbm{1}_{1}(u^{n}) \mathbbm{1}_{M^n+1} (k) \mu_{t_n,a}^{pos}(B_1) \label{p_3} \\
&P(s_k^{n+1} = s_k^n \mid X^n) &&= \mathbbm{1}_0(u^{n}) + \mathbbm{1}_1(u^{n}) (1-\mathbbm{1}_{M^n +1}(k)) + \mathbbm{1}_{-1} (u^{n}) (1 - \mathbbm{1}_{l^n}(k)) \label{s_1}\\
&P(s_k^{n+1} = 0 \mid X^n) &&= \mathbbm{1}_{-1}(u^{n}) \mathbbm{1}_{l^n}(k)  \label{s_2} \\
&P(s_k^{n+1} \in B_2 \mid X^n) &&= \mathbbm{1}_{1}(u^{n}) \mathbbm{1}_{M^n+1} (k) \mu^{s}(B_2) \label{s_3} 
\end{alignat}
\begin{alignat}{2}
&P(c_k^{n+1} = c_k^n \mid X^n) &&= \mathbbm{1}_0(u^{n}) + \mathbbm{1}_1(u^{n}) (1-\mathbbm{1}_{M^n+1}(k)) + \mathbbm{1}_{-1} (u^{n}) (1 - \mathbbm{1}_{l^n}(k)) \label{c_1} \\
&P(c_k^{n+1} = 0 \mid X^n) &&= \mathbbm{1}_{-1}(u^{n}) \mathbbm{1}_{l^n}(k)   \label{c_2} \\
&P(c_k^{n+1} \in B_3 \mid X^n) &&= \mathbbm{1}_{1}(u^{n}) \mathbbm{1}_{M^n+1} (k) \mu^{cap}(B_3) \label{c_3}
\end{alignat}
for all $k = 1,...,K$ and $B_1, B_2, B_3 \in \mathcal{B}(\mathbb{R})$.  
In the equations $(\ref{p_1})$ - $(\ref{p_3})$ the transition of the position of the accidents from time step $n$ to $n+1$ is modeled. The first case $(\ref{p_1})$ is the steady state where either no event occurs ($u^n=0$) or the considered line does not correspond to the one in which a new accident would be denoted ($u^n=1$) or an old one dissolved ($u^n = -1$). The dissolution of an accident in a certain line is represented in the second expression $(\ref{p_2})$ whereas the position of a new accident is introduced in the third equation $(\ref{p_3})$. In the case where an accident dissolves we check for the variables $u^n$ (which event occurs) and $l^n$ (which accident potentially dissolves). In the latter we used the probability measure $\mu_{t_n,ac}^{pos}$ from $(\ref{mupos})$ in the $(n+1)$st accident entry to determine a new accident's position.

The treatment of the accident size variables $(\ref{s_1})$ - $(\ref{s_3})$ follows exactly the same pattern as with the accident positions. The same applies to the capacity reduction in $(\ref{c_1})$ - $(\ref{c_3})$. 

\subsection{Macroscopic Model}\label{sec:MakroModelll}
In contrast to the microscopic model, in the macroscopic model we do not describe the behavior of each vehicle. The variable of interest is the traffic density function and the dynamics are based on the LWR model (e.g. \cite{LWR}).
Basically the macroscopic model can be described by a conservation law with the following space dependent flux
\begin{equation}
\label{Flux_macro}
F_{ac}(x, \rho) = c_{ac}(x) c_{road} (x) f(\rho),
\end{equation}
where $f: [0,1] \rightarrow [0,\infty)$ is a twice continuously differentiable LWR type function  (i.e. $f(\rho)= \rho v(\rho)$ with $v(\rho) = 1-\rho$) that satisfies:
\begin{equation}
\label{condonf}
f(0) = 0 = f(1),~~~~  f'' \leq c < 0,  \text{ for some } c \leq 0,~~~~ \exists !  ~ \rho^* \in (0,1) \text{ s.t. } f'(\rho^*) = 0.
\end{equation}
Similarly to the microscopic model the function $c_{road}: \mathbb{R} \rightarrow \mathbb{R}_+$ represents the general road capacity. We use the function $c_{ac}: \mathbb{R} \rightarrow \mathbb{R}_+ $ to model the accidents in the traffic model. The index $ac$ provides information about the current active accidents. In the case of several accidents, say $M \in \mathbb{N}_0$, we assign each of them an index $j \in \lbrace 1,..., M \rbrace$. For accident $j$ we denote the position $p_j \in \mathbb{R}$, the size $s_j \in \mathbb{R}_+$ and the capacity reduction by $c_j \in [0,c_{max}]$, where  $c_{max} \in [0,1)$.  Then $c_{ac}$ is given by
\begin{equation*}
c_{ac} (x) = \prod_{j=1}^M \Big(1 - c_j \mathbbm{1}_{[p_j - \frac{s_j}{2}, p_j + \frac{s_j}{2}]} (x)\Big).
\end{equation*}
All together we end up with the following Cauchy problem
\begin{align}
\begin{split}
\label{makroProblem}
\rho_t + (F_{ac}(x, \rho))_x &= 0, ~~\rho(x,0) = \rho_0 (x),
\end{split}
\end{align}
where $\rho_0$ describes the initial density on the road. 
Under some assumptions on these functions we obtain a unique entropy solution to the Cauchy problem $(\ref{makroProblem})$ in the sense of a function in $BV(\mathbb{R})$. As it turns out we achieve this entropy solution if we demand that 
\begin{itemize}
\item $f$ is a LWR flux, i.e. it fulfills the conditions $(\ref{condonf})$ and $f, f' \in L^{\infty}(\mathbb{R})$
\item $\rho_0 \in BV(\mathbb{R})$, ~ $c_{ac}(x) c_{road}(x) \in C^2(\mathbb{R}) \cap TV(\mathbb{R}) \cap L^{\infty}(\mathbb{R})$
\item $(c_{ac}(x) c_{road}(x))' \in L^{\infty}(\mathbb{R}) \cap L^1(\mathbb{R}), ~(c_{ac}(x) c_{road}(x))'' \in L^1(\mathbb{R})$
\end{itemize}
For a proof we refer to \cite{A}. 

This section is proceeded quite similarly to what was done for the microscopic model. First we are interested in where an accident is going to occur. We consider two different scenarios for an accident. Accidents according to type 1 are caused by a high value of the flux $(\ref{Flux_macro})$. Therefore, we define a time and accident dependent constant
\begin{equation*}
C_{F_{ac}}^{Mac}(t) = \int_a^b F_{ac}(x,\rho (x,t)) dx,
\end{equation*}
where $a<b$ and are also allowed to be $-\infty$ and $\infty$. 
The integral is well defined if $F_{ac}$ fulfills the conditions for the existence of an entropy solution since
\begin{align*}
|C_{F_{ac}}^{Mac}(t)| &\leq \int_a^b | F_{ac}(x,\rho (x,t)) | dx \leq \| c_{road} \|_{\infty} \| v \|_{\infty} \int_a^b \rho(x,t) dx.
\end{align*}
For any $B \in \mathcal{B}([a,b])$ and $\rho \in BV([a,b] \times \mathbb{R}_+)$, this allows to define the following family of probability measures $(\mu_t^{F_{ac},Mac})_{t\geq 0}$ on $([a,b],\mathcal{B}([a,b]))$ by
\begin{equation}
\label{muFMak}
\mu_t^{F_{ac},Mac} (B) =  \frac{\int_B F_{ac}(x,\rho(x,t))dx}{C_{F_{ac}}^{Mac}(t)}.
\end{equation}
The latter measure models the probability of having an accident of type 1 in the given Borel set. 

Accidents of type 2 were modelled due to tailbacks of traffic jams in the microscopic model using increases in the local density functions. The corresponding way to model this in a macroscopic model, would be by considering the derivative of the traffic density function. Since we have to work with weak solutions and functions of bounded variation (BV) it is not possible to use classical derivatives. Fortunately, as shown for example in \cite{BVfct}, derivatives of BV functions correspond to signed Radon measures $D \rho$  given by the total variation $TV_\rho$ on $B \in \mathcal{B}([a,b])$
\begin{equation*}  
|D \rho|^{Mac}(B,t) = \sup \bigg\lbrace \int_B \rho(x,t) \phi'(x) dx \mid \phi \in C_c^1(B), |\phi| \leq 1 \bigg\rbrace.
\end{equation*}
We can split this measure up into a positive and a negative part using the Hahn decomposition (see e.g. \cite{BVfct}) constituting an increase or a decrease in the density respectively by
\begin{equation*}
|D \rho|^{Mac}(B,t) = D \rho^{+,Mac} (B,t) + D \rho^{-,Mac} (B,t).
\end{equation*}
We are particularly interested in the positive part since we want to detect tailbacks of traffic jams that can be characterized by increasing values of the flux. This leads us to the following family of probability measures $(\mu_{t}^{D, Mac})_{t \geq 0}$ on $([a,b], \mathcal{B}([a,b]))$
\begin{equation}
\mu_{t}^{D, Mac} (B) = \frac{D \rho^{+,Mac} (B,t)}{D \rho^{+,Mac}([a,b],t)}
\end{equation} 
for the probability of an accident of type 2 in $B \in \mathcal{B}([a,b])$. \\
If we assume that accidents of type 1 and type 2 occur with a probability of $\beta^{Mac} \in [0,1]$ and $1-\beta^{Mac}$ respectively, we end up with the following overall family of probability measures $(\mu_{t,{ac}}^{pos,Mac})_{t\geq 0}$ on $([a,b], \mathcal{B}([a,b]))$ for the position of an accident
\begin{equation}
\mu_{t,{ac}}^{pos,Mac} (B) = \beta^{Mac} \mu_t^{F_{ac},Mac} (B) + (1- \beta^{Mac}) \mu_t^{D,Mac} (B)
\end{equation}
for $B \in \mathcal{B}([a,b])$.

After having investigated the position of an accident, we will now care about whether an accident happens. Therefore, we consider a stochastic process containing all necessary information of our system. As in \cite{A} we define a state space for the stochastic process
\begin{equation*}
E^{Mac} = \mathbb{R}^{\mathbb{N}}  \times \mathbb{R}^{\mathbb{N}} \times [0,1]^{\mathbb{N}} \times BV(\mathbb{R}),
\end{equation*}
and set $\mathcal{E}^{Mac} = \sigma(E^{Mac})$ the smallest $\sigma$-algebra over $E^{Mac}$. Here, $E^{Mac}$ models the whole state of the accident model and consists of variables $E \ni  z = (p,s,c,\rho)$. The position of accident $j$ is denoted in $p_j$, whereas the size and capacity reduction are expressed by $s_j$ and $c_j$. The traffic density $\rho$ denotes the weak unique entropy solution of $(\ref{makroProblem})$.

We set $\mu^{s,Mac}$ and $\mu^{cap,Mac}$ as two probability measures on $(\mathbb{R},\mathcal{B}(\mathbb{R}))$ modeling the sizes and capacity reductions of accidents, respectively.

Furthermore, let $\lambda_R^{Mac}>0$ be the rate of dissolving an accident, $\lambda_F^{Mac}>0$ the rate of an accident type 1 and $\lambda_D^{Mac}>0$ the rate of an accident type 2. Then we define the overall rate for the occurrence of an accident for a given state $z \in E^{Mac}$ as
\begin{equation*}
\lambda_A^{Mac}(z,t) = \lambda_F^{Mac} C_{F_{ac}}^{Mac}(t) + \lambda_D^{Mac} D\rho^+(\mathbb{R},t). 
\end{equation*}
Also allowing for the dissolution of an accident the rate of an event (new accident or dissolution of an accident) for a given state $z \in E^{Mac}$ is defined by
\begin{equation*}
\psi(z,t) = \lambda_F^{Mac} C_{F_{ac}}^{Mac}(t) + \lambda_D^{Mac} D\rho^+(\mathbb{R},t) + \lambda_R^{Mac} M(t), 
\end{equation*}
where $M(t) \in \mathbb{N}$ describes the number of currently active accidents in $z$. Again $ac$ is dependent on the state of the process.
To include and exclude accidents in our model, we define two functions
\begin{align*}
& m(c) = \min \lbrace i \in \mathbb{N} \mid c_i=0 \rbrace \\
&\pi_i (u,v) = (v_1,...,v_{i-1},u,v_{i+1},...) \in \mathbb{R}^{\mathbb{N}}.
\end{align*}
For $z \in E^{Mac}$ and $B \in \mathcal{E}^{Mac}$ the transition probability of moving from state $z$ into any state of $B$ can be given as
\begin{align*}
\eta(z,B) =& \frac{1}{\lambda_R^{Mac} \sum_{i \in \mathbb{N}} \mathbbm{1}_{\lbrace c_i>0 \rbrace} + \lambda_A^{Mac}} \Big(\lambda_R^{Mac} \sum_{i \in \mathbb{N}} \mathbbm{1}_{\lbrace c_i>0 \rbrace} ~ \epsilon_{(p,s,\pi_i(0,c),\rho)}(B) \\
&+ \lambda_A^{Mac} \int_{\mathbb{R}^2 \times [0,1)} \epsilon_{(\pi_{m(c)} (\tilde{p},p),\pi_{m(c)} (\tilde{s},s),\pi_m{(c)}(\tilde{c},c),\rho)} (B)\\
& d(\mu^{pos,Mac} \otimes \mu^{s,Mac} \otimes \mu^{cap, Mac})(\tilde{p}, \tilde{s}, \tilde{c}) \Big).
\end{align*}
Between the stochastic jumps given by new events, the system evolves in a deterministic way. We denote the deterministic evolution of the system as
\begin{equation*}
\phi_t(p_0,s_0,c_0, \rho_0) = (p_0,s_0,c_0, \rho(t)),
\end{equation*}
where $\rho$ is the unique entropy solution to $(\ref{makroProblem})$ and $p_0,~s_0,~c_0$ and $\rho_0$ are given initial data.
To construct the stochastic process, we now use the idea of a thinning algorithm similarly applied as in  \cite{15} to construct a sequence of event times $(T_n)_{n\in \mathbb{N}}$. For $t_n \in [0,T]$ and $z_n \in E^{Mac}$ we then have
\begin{align*}
\begin{split}
P(T_{n+1} \leq t) &= 1 - e^{- \int_{t_n}^t \psi(\phi_{\tau-t_n}(z_n),\tau)d\tau}\\
P(Z_{n+1} \in B | T_{n+1} = t) &= \eta(\phi_{t-t_n}(z_n), B)
\end{split}
\end{align*}
for $t\geq t_n$ and $B \in \mathcal{E}^{Mac}$. 

Let us now define the stochastic process $X =(X(t))_{t\in [0,T]}$ having values in $E^{Mac}$ as
\begin{equation*}
X(t) = Z_n \text{,   for } t \in [T_n, T_{n+1}),
\end{equation*}
where $(T_n)_{n \in \mathbb{N}}$ and $(Z_n)_{n \in \mathbb{N}}$ are a sequence of event times and of states generated by the transition measure $\eta$, respectively.

\subsection{A Microscopic Model with Macroscopic Accidents} \label{TildeModel}
In the following sections we will be interested in whether we obtain some convergence of the microscopic model against the macroscopic model if we increase the number of vehicles on the road. To show such a convergence, we create another third model in between, which still will be of a microscopic nature but have accidents according to the corresponding macroscopic model. This allows for a separation of the proof. One step will be to show convergence of the two microscopic models, where we have to treat the randomness of the occurrence of accidents. In the second part of the proof any stochastic components can be left out, and we only have to care about the micro-macro limit of the microscopic model with macroscopic accidents and the macroscopic model (see Section $\ref{sec:ConvergenceMikroMakro}$). All variables concerned with this third model will be marked with a tilde symbol. 

Therefore, let us define a state space for the third stochastic process using the state space $E^{Mac}$ from the macroscopic model 
\begin{equation*}
\tilde{E}^{N} = E^{Mac} \times \mathbb{R}^{N},
\end{equation*}
depending on the total number of vehicles in the system.
We denote $\tilde{E}^N \ni \tilde{y} = (y, \tilde{x})$ where $y$ is taken from the macroscopic state space and $\tilde{x}$ is the vector of the positions of all $N$ cars, where the $i$-th entry represents the position of vehicle $i$.
Let us define a discrete time stochastic process $\tilde{X} = (\tilde{X}_n)_{n \in \mathbb{N}}$ on a time grid $(t_n)_{n \in \mathbb{N}}$ satisfying $\Delta t = t_{j+1} - t_j$, $j \in \mathbb{N}$,  with
\begin{equation*}
\tilde{X}_n = (X^{Mac} (t_n), \tilde{x}^n).
\end{equation*} 
For the vehicle position $\tilde{x}$ we consider similar equations as in Section $\ref{sec:BasicModel}$ but now using the accidents from the macroscopic model. To be specific the system of ordinary differential equations for the positions of the cars are now given by
\begin{align}
\begin{split}
\label{AusgangsODE}
\dot{\tilde{x}}_i(t) &= c_{road}(\tilde{x}_i(t)) \tilde{c}_{ac}(\tilde{x}_i(t)) v\bigg(\frac{L}{\tilde{x}_{i+1}(t)-\tilde{x}_i(t)}\bigg), ~~~ i=1,...,N-1  \\
\dot{\tilde{x}}_N(t) &= c_{road}(\tilde{x}_N(t)) \tilde{c}_{ac}(\tilde{x}_N(t)) v\bigg(\frac{L}{\tilde{x}_1 - \tilde{x}_N + (b-a)}\bigg).
\end{split}
\end{align}
Assume that $M$ denotes the number of currently active accidents in the macroscopic model. Then the function that regulates the capacity reductions due to accidents is given by
\begin{equation*}
\tilde{c}_{ac}: \mathbb{R} \rightarrow \mathbb{R}, ~~~ x \mapsto \prod_{j=1}^{M} \Big( 1-c_j^{Mac} \mathbbm{1}_{\big[p_j^{Mac}-\frac{s^{Mac}_j}{2}, p_j^{Mac}+\frac{s^{Mac}_j}{2}\big]} (x) \Big),
\end{equation*}
where the accident parameters $p_j^{Mac}, s_j^{Mac}, c_j^{Mac}$ for the position, size and capacity reduction are the ones taken from the macroscopic model for accident $j$. Instead of mentioning $p_j^{Mac}$, $s_j^{Mac}$ and $c_j^{Mac}$, we use the notation of an index $ac$ that represents the accident situation.
The transition probabilities for the vehicle positions $\tilde{x}$ for $i=1,...,N-1$ on the given time grid are computed by
\begin{alignat}{1}
\begin{split}
\label{x_1Tilde}
&P\Big(\tilde{x}_i^{n+1} = \tilde{x}_i^n + \Delta t c_{road}(\tilde{x}_i^n)  \tilde{c}_{ac}(\tilde{x}_i^n)  v \bigg( \frac{L(N)}{\tilde{x}_{i+1}(t)-\tilde{x}_i(t)} \bigg) \mid \tilde{X}^n \Big) = 1  \\
&P\Big(\tilde{x}_N^{n+1} = \tilde{x}_N^n + \Delta t c_{road}(\tilde{x}_N^n)  \tilde{c_{ac}}(\tilde{x}_N^n)  v \bigg( \frac{L(N)}{\tilde{x}_{1}(t)-\tilde{x}_N(t) + (b-a)} \bigg) \mid \tilde{X}^n\Big) = 1. 
\end{split}
\end{alignat}
All other transition conditions are not relevant for our microscopic model with macroscopic accidents since the accident situation is fully governed by the corresponding macroscopic model.

\section{Convergence Analysis}\label{sec:ConvergenceMikroMakro}
In this section, we analytically investigate the behavior of the local densities from the microscopic model with macroscopic accidents if we increase the number of vehicles, i.e. $N \rightarrow \infty$. Especially, we want to compare this limit to the traffic density function from the macroscopic model. Several approaches to this micro-macro limit have been considered (e.g. \cite{MicMac1, MicMac3,FollowLeader,MicMac2}). Most of them admit only for flux functions that do not depend on an additional spatial component. Therefore, we use ideas from \cite{FollowLeader} and proceed as follows:

First, we carry out a coordinate transform to Lagrangian coordinates and set up the definition of local inverse densities. These are required because, unfortunately, we are not directly able to show convergence for the local densities. For those local inverse densities, we set up a Cauchy problem which will be related to the Cauchy problem from the macroscopic model. Afterwards, we show that the local inverse densities converge to a weak solution of the related Cauchy problem. In the end, one can prove a correspondence between the weak solutions of these two Cauchy problems.

\subsection{Lagrangian Coordinates and Derivation of a Lax-Friedrichs Scheme}
So far, the microscopic model was considered in Eulerian coordinates in $(\ref{odeAcc})$. We now introduce the connected Lagrangian setting. In Eulerian coordinates we assigned a unique index to each vehicle on the road. Depending on the time we denoted the position of each vehicle by the variable $x_i(t)$. 

Another way to consider the setting would be by taking $y \in (0,1)$ as a continuous number of a vehicle. This so-called Lagrangian variable depends on a given position and the time, such that $y(x,t)$ gives the infinitesimal number of the vehicle that is located at $x$ at time $t$. On the other hand we are able to understand the Eulerian coordinate $x$ depending on the Lagrangian coordinate $y$ and the time, such that $x(y,t)$ represents the position of vehicle $y$ at time $t$. A further and more formal derivation of the connection between Eulerian and Lagrangian coordinates can be found in Theorem $\ref{eqvSOL}$.

We introduce some additional definitions for the microscopic model with macroscopic accidents. 
On a road with $N \in \mathbb{N}$ vehicles we define the local density of vehicle $i=1,...,N-1$ in the new microscopic model with accidents depending on the macroscopic model at time $t>0$ to be
\begin{align}
\begin{split}
\label{locDensityTilde}
\tilde{\rho}_i^{(N)}(t) &= \frac{L(N)}{\tilde{x}_{i+1}(t) - \tilde{x}_i (t)}, ~~
\tilde{\rho}_N^{(N)}(t) = \frac{L(N)}{\tilde{x}_{1}(t) - \tilde{x}_N (t) + b-a}.
\end{split}
\end{align}
For the inverse of the local density function, let us define for $i=1,...,N$
\begin{equation*}
w_i^{(N)}(t) = \frac{1}{\tilde{\rho}_i^{(N)}(t)}
\end{equation*}
as well as a Lagrangian velocity function $\tilde{v}: (0,1] \rightarrow \mathbb{R}$ with $\tilde{v}(x) = v(\frac{1}{x})$. Assuming $v(\rho) = \max \lbrace 0,1-\rho \rbrace$, $\tilde{v}$ is bounded for arguments $x\geq 1$ and Lipschitz continuous. To ensure mass conservation in the limit we allow the vehicle length $L$ to be dependent on $N$.
Throughout this section we additionally assume that there exist $\varepsilon>0$ and $1 + \varepsilon \leq K < \infty$ such that 
\begin{align}
\label{Bv-cond}
\begin{split}
& 1 +  \varepsilon \leq w_i^{(N)} (0) \leq K,~~~ i=1,...,N\\ 
& \sum_{i=1}^{N-1}  |w_{i+1}^{(N)} (0) - w_i^{(N)} (0)| + |w_1^{(N)}(0) - w_N^{(N)}(0)| \leq K.
\end{split}
\end{align}
Both assumptions seem quite reasonable. The first ensures that all inverses of the local densities are uniformly bounded at the initial time. Additionally, we require some kind of safety distance between the vehicles and assume that they do not start in a bumper-to-bumper situation. The second assumption is basically a bound on the total variation on the inverses of the local densities at initial time.  

We now introduce the Lagrangian way of describing the positions of the vehicles. It can be used to construct a Lax-Friedrichs-type sequence of local inverse density functions.
In Eulerian coordinates we define the function $c: \mathbb{R} \times \mathbb{R}_+ \rightarrow \mathbb{R}$
\begin{equation*}
c(x,t) = \tilde{c}_{ac}(x) c_{road}(x).
\end{equation*}
The time dependency of this function is incorporated  via the accident dependency.

We define a Lagrangian grid $\big\lbrace y_{i-\frac{1}{2}} \big\rbrace_{i=1}^N$ where $y_{i-\frac{1}{2}}=(i-1)L(N)$. To adjust the capacity function, we define $\tilde{c}: [0,1] \times \mathbb{R}_+ \rightarrow \mathbb{R}$ as the Lagrangian capacity function. On a grid point of the Lagrangian grid it is defined for $i=1,...,N$ by 
\begin{equation*}
\tilde{c}(y_{i-\frac{1}{2}},t) = c(x_i(t),t).
\end{equation*}
Between the grid points we define the function using a linear interpolation:
\begin{equation*}
\tilde{c}(y,t) = c\big((y-y_{i-\frac{1}{2}})x_{i+1}(t) + (y_{i+\frac{1}{2}} - y) x_i(t) \big) \text{,     for } y \in [y_{i-\frac{1}{2}}, y_{i+\frac{1}{2}}].
\end{equation*}
Using these definitions for the local inverse densities, the initial system of ordinary differential equations describing the position of a vehicle for $i=1,...,N-1$ transforms into
\begin{align*}
\dot{w}_i^{(N)} (t) =& \frac{\dot{\tilde{x}}_{i+1}(t) - \dot{\tilde{x}}_i(t)}{L(N)} = \frac{c(\tilde{x}_{i+1}(t), t)v(\tilde{\rho}_{i+1}(t)) - c(\tilde{x}_i(t), t)v(\tilde{\rho}_i(t))}{L(N)} \\
\overset{(\ref{AusgangsODE})}{=} &\frac{\tilde{c}(y_{i+\frac{1}{2}}, t)\tilde{v}(w_{i+1}^{(N)}(t)) - \tilde{c}(y_{i-\frac{1}{2}}, t)\tilde{v}(w_i^{(N)}(t))}{L(N)}.
\end{align*}
Equivalently for the $N$-th vehicle we have
\begin{equation*}
\dot{w}_N^{(N)} (t) = \frac{\tilde{c}(y_{\frac{1}{2}}, t)\tilde{v}(w_{1}^{(N)}(t)) - \tilde{c}(y_{N-\frac{1}{2}}, t)\tilde{v}(w_N^{(N)}(t))}{L(N)}.
\end{equation*}
The solutions to this system of ordinary differential equations can be used to develop a numerical scheme for the local inverse densities for the approximation of a weak solution of $(\ref{cpIv})$. 

Consider an equidistant time grid $(t_j)_{j \in \mathbb{N}}$, with $\Delta t = t_{j+1} - t_j$. Using a modified Euler approximation we approximate for $i=2,...,N-1$
\begin{equation}
\label{approxLxFstyle}
\tilde{x}_i(t_{j+1}) = \frac{1}{2}(\tilde{x}_{i+1}(t_j) + \tilde{x}_{i-1}(t_j)) + \Delta t~ c(\tilde{x}_i(t_j),t_j) v\bigg(\frac{L(N)}{\tilde{x}_{i+1}(t_j) - \tilde{x}_i(t_j)} \bigg).
\end{equation}
For the first and the last vehicle we add 
\begin{align}
\begin{split}
\label{approxLxFstyle1}
\tilde{x}_1(t_{j+1}) =& \frac{1}{2}(\tilde{x}_{2}(t_j) + \tilde{x}_{N}(t_j)-(b-a)) + \Delta t~ c(\tilde{x}_1(t_j),t_j) v\bigg(\frac{L(N)}{\tilde{x}_{2}(t_j) - \tilde{x}_1(t_j)} \bigg) \\
\tilde{x}_N(t_{j+1}) =& \frac{1}{2}(\tilde{x}_{1}(t_j) + \tilde{x}_{N-1}(t_j)+(b-a)) + \Delta t~ c(\tilde{x}_N(t_j),t_j) v\bigg(\frac{L(N)}{\tilde{x}_{1}(t_j) - \tilde{x}_N(t_j)+(b-a)} \bigg).
\end{split}
\end{align}
Using this for the local inverse density in the Lagrangian setting and for $i=2,...,N-1$, $j=0,1,2,...$ we get
\begin{align*}
w_i^{(N)} (t_{j+1}) =& \frac{1}{\rho_i^{(N)}(t_{j+1})} = \frac{\tilde{x}_{i+1}(t_{j+1}) - \tilde{x}_i(t_{j+1})}{L(N)} \nonumber \\
\overset{(\ref{approxLxFstyle})}{=}&\frac{1}{L(N)} \bigg[\frac{\tilde{x}_{i+2}(t_j) + \tilde{x}_{i}(t_j)}{2} + \Delta t~ c(\tilde{x}_{i+1}(t_j),t_j) \tilde{v}(w_{i+1}^{(N)}(t_j)) \nonumber \\
&-\Big(\frac{\tilde{x}_{i+1}(t_j) + \tilde{x}_{i-1}(t_j)}{2}+ \Delta t~ c(\tilde{x}_i(t_j),t_j) \tilde{v}(w_{i}^{(N)}(t_j))\Big)\bigg] \nonumber \\
=& \frac{w_{i+1}^{(N)} (t_j) + w_{i-1}^{(N)} (t_j)}{2} + \frac{\Delta t}{L(N)} \Big(\tilde{c}(y_{i+\frac{1}{2}},t_j) \tilde{v}(w_{i+1}^{(N)}(t_j)) - \tilde{c}(y_{i-\frac{1}{2}},t_j) \tilde{v}(w_{i}^{(N)}(t_j))  \Big).\nonumber 
\end{align*}
Let us make one further approximation and substitute 
\begin{align*}
&\tilde{c}(y_{i-\frac{1}{2}},t_j) \tilde{v}(w_{i}^{(N)}(t_j)) = \frac{\tilde{c}(y_{i+\frac{1}{2}},t_j) \tilde{v}(w_{i+1}^{(N)}(t_j)) + \tilde{c}(y_{i-\frac{3}{2}},t_j) \tilde{v}(w_{i-1}^{(N)}(t_j))}{2}.
\end{align*}
Then, we obtain an expression known from the Lax-Friedrichs scheme by
\begin{align}
\begin{split}
\label{H}
w_i^{(N)} (t_{j+1}) =&  \frac{w_{i+1}^{(N)} (t_j) + w_{i-1}^{(N)} (t_j)}{2}  + \frac{\Delta t}{2L(N)} \Big(\tilde{c}(y_{i+\frac{1}{2}},t_j) \tilde{v}(w_{i+1}^{(N)}(t_j)) - \tilde{c}(y_{i-\frac{3}{2}},t_j) \tilde{v}(w_{i-1}^{(N)}(t_j))\Big).
\end{split}
\end{align}
Similar computations can be made for the first and the last vehicle which yield
\begin{align}
\begin{split}
\label{H2}
w_1^{(N)} (t_{j+1}) =&  \frac{w_{2}^{(N)} (t_j) + w_{N}^{(N)} (t_j)}{2} + \frac{\Delta t}{2L(N)} \Big(\tilde{c}(y_{\frac{3}{2}},t_j) \tilde{v}(w_{2}^{(N)}(t_j)) - \tilde{c}(y_{N-\frac{1}{2}},t_j) \tilde{v}(w_{N}^{(N)}(t_j))\Big),  \\
w_N^{(N)} (t_{j+1}) =&  \frac{w_{1}^{(N)} (t_j) + w_{N-1}^{(N)} (t_j)}{2}  + \frac{\Delta t}{2L(N)} \Big(\tilde{c}(y_{\frac{1}{2}},t_j) \tilde{v}(w_{1}^{(N)}(t_j)) - \tilde{c}(y_{N-\frac{3}{2}},t_j) \tilde{v}(w_{N-1}^{(N)}(t_j))\Big).
\end{split}
\end{align}
Only information from the microscopic model with macroscopic accidents was used to construct this scheme. As we have been working with single vehicles so far, we define a piecewise constant function for the initial density by
\begin{equation}
\label{defw0}
w_0^{(N)}(y)=w_i^{(N)}(0) \text{,   for } y \in [y_{i-\frac{1}{2}}, y_{i+\frac{1}{2}}).
\end{equation}

\subsection{Convergence to a Weak Solution}
Under some assumptions we are going to show that such a Lax-Friedrichs scheme approximates a weak solution $w: [0,1] \times \mathbb{R}_+ \rightarrow \mathbb{R}$ of the following conservation law
\begin{align}
\begin{split}
\label{cpIv}
w_t - (\tilde{c}(y,t) \tilde{v}(w))_y &= 0 ,~~~w(0,y) = w_0(y) 
\end{split}
\end{align}
for some initial local inverse density $w_0$. 
First, we present a result that states under which conditions we achieve convergence of a sequence to a weak solution of a Cauchy problem having a space and time dependent flux function. Afterwards, we develop some results that will help to apply Theorem $\ref{Thm45}$ to our setting.

\begin{theorem}[Theorem 4.5 in \cite{B}]
\label{Thm45}
Consider the conservation law
\begin{align}
\label{csvlaw2}
\begin{split}
u_t - f(c(x,t),u)_x &= 0,~~~ t>0, ~ x \in \mathbb{R}, ~~~u(0,x) = u_0(x).
\end{split}
\end{align}
Assume that for the discretization parameters in space $\Delta x >0$ and time $\Delta t>0$ it holds
\begin{enumerate}
\item $u_0 \in L^\infty(\mathbb{R}), ~~~ a \leq u_0(x) \leq b,$  for almost all  $x \in \mathbb{R}$,
\item $c \in L^\infty (\mathbb{R} \times \mathbb{R}_+) \cap BV_{loc}(\mathbb{R} \times \mathbb{R}_+)$,     $\alpha \leq c(x,t) \leq \beta$  for almost all  $(x,t) \in \mathbb{R} \times \mathbb{R}_+$, 
\item $u \mapsto f(c,u) \in C^2([a,b])$ , for all  $c \in [\alpha, \beta]$, 
\item $c \mapsto f(c,u) \in C^1([\alpha, \beta])$ , for all  $u \in [a,b]$, 
\item for almost all $(x,t) \in \mathbb{R} \times \mathbb{R}_+:~ \frac{\partial^2}{\partial u^2} f(c(x,t),u) \neq 0$  for almost all  $u \in [a,b]$,
\item the Lax-Friedrichs scheme approximation stays uniformly bounded,
\item $\lambda L_{f_u} \leq 1- \kappa, ~~~ \lambda = \frac{\Delta t}{\Delta x},$  for some $\kappa \in (0,1)$,
\end{enumerate}
where $L_{f_u}$ denotes the Lipschitz constant of $f$ in the second argument.

Discretize the time domain $\mathbb{R}_+$ via $t_j = j \Delta t,~ j \in \mathbb{N}_0$ and the spatial component $\mathbb{R}$ by $x_i = i \Delta x,~ i \in \mathbb{Z}$. Define $c_i^j = \underset{x \searrow x_i} {\lim} c(x,\hat{t}_j)$ for any $\hat{t}_j \in [t_j, t_{j+1})$ for which the limit exists. Then, an one step approximation of the Lax-Friedrichs scheme writes for $j>0$
\begin{equation*}
u_i (t_{j+1}) = \frac{1}{2}(u_{i+1} + u_{i-1}) + \frac{\lambda}{2} \Big(f(c_{i+1}^j,u_{i+1}(t_j)) - f(c_{i-1}^j, u_{i-1}(t_j))\Big). 
\end{equation*}
For the initial time we set
\begin{equation*}
u_i (t_0) = \frac{1}{2 \Delta x} \int_{x_{i-1}}^{x_{i+1}} u_0(x) dx.
\end{equation*}
Using a staggered form of the Lax-Friedrichs scheme, we define the piecewiese constant function $u^\Delta$ for $i+j$ being even as
\begin{equation*}
u^{\Delta} (t,y) = u_i (t_j),~~~ \text{for } (t,y) \in [t_j, t_{j+1}) \times [x_{i-1}, x_{i+1}).
\end{equation*}
Passing if necessary to a subsequence, we have $u^\Delta \rightarrow u$
as $\Delta t \rightarrow 0$ and $\Delta x \rightarrow 0$ in $L_{loc}^p (\mathbb{R}_+ \times \mathbb{R})$ for any $p<\infty$ and where $u \in L^\infty(\mathbb{R}_+ \times \mathbb{R})$ is a weak solution to $(\ref{csvlaw2})$.
\end{theorem}

We now try to adapt this theorem to the setting in $(\ref{cpIv})$. The only challenging condition will be to show the uniform bound on the Lax-Friedrichs approximation. Therefore, we introduce the following lemma.
\begin{lemma}
\label{boundednessLemma}
Assume $c \in C^{0,1}(\mathbb{R})$ uniformly bounded and $v \in C^1(\mathbb{R}_+)$. Set $f(y,w) = c(y) v(w)$. If $(\ref{Bv-cond})$ holds and the CFL condition is satisfied, the Lax-Friedrichs approximations from $(\ref{H})$ and $\ref{H2}$ are uniformly bounded, setting $v(w) = 1 - \frac{1}{w}$. 
\end{lemma}
\begin{proof}

Using $(\ref{Bv-cond})$ $w_0$ initially is bounded by $K = \| w^{(N)}(0) \|_\infty$ from above. Set $\lambda = \frac{\Delta t}{\Delta y}$ where in the Lagrangian grid $\Delta y$ just corresponds to the length of one vehicle. We calculate the one step increase of the Lax Friedrichs bound
\begin{align*}
|w_i^{(N)} (t_{j+1})| =& \Big|\frac{w_{i-1}^{(N)}(t_j) + w_{i+1}^{(N)}(t_j)}{2} + \frac{\lambda}{2} \Big(f(y_{i+\frac{1}{2}}, w_{i+1}^{(N)}(t_j)) - (f(y_{i-\frac{3}{2}}, w_{i-1}^{(N)}(t_j))\Big) \Big|\\
=& \Big|\frac{w_{i-1}^{(N)}(t_j) + w_{i+1}^{(N)}(t_j)}{2} + \frac{\lambda}{2} \Big(f(y_{i+\frac{1}{2}}, w_{i+1}^{(N)}(t_j)) - f(y_{i+\frac{1}{2}}, w_{i-1}^{(N)}(t_j)) \\
&+ f(y_{i+\frac{1}{2}}, w_{i-1}^{(N)}(t_j))  - (f(y_{i-\frac{3}{2}}, w_{i-1}^{(N)}(t_j))\Big) \Big| \\
=& \Big|\frac{w_{i-1}^{(N)}(t_j) + w_{i+1}^{(N)}(t_j)}{2}  + \frac{\lambda}{2} \Big(f_w(y_{i+\frac{1}{2}}, \xi)(w_{i+1}^{(N)}(t_j) - w_{i-1}^{(N)}(t_j))\\
&  + f(y_{i+\frac{1}{2}}, w_{i-1}^{(N)}(t_j))  - (f(y_{i-\frac{3}{2}}, w_{i-1}^{(N)}(t_j))\Big)\Big|\\
\leq& \Big| \frac{1}{2} \Big( w_{i-1}^{(N)}(t_j) (1- \lambda f_w(y_{i+\frac{1}{2}}, \xi) + w_{i+1}^{(N)}(t_j) (1 + \lambda f_w(y_{i+\frac{1}{2}}, \xi)) \Big) \Big|  \\
&+ \Big|\Delta t L_c \frac{y_{i+\frac{1}{2}} - y_{i-\frac{3}{2}}}{2 \Delta y}\Big| \\
\leq & \|w^{(N)} (t_j) \|_\infty + \Delta t L_c 
\end{align*}
for some $\xi \in \mathbb{R}$ being a convex combination of $w_{i-1}^{(N)}(t_j)$ and $w_{i+1}^{(N)}(t_j)$. Here, $L_c$ denotes the Lipschitz constant of $c$ and $f_w$ the partial derivative of $f$ with respect to $w$. Then, for a finite time horizon $T >0$ we get
\begin{align*}
\| w^{(N)}(T) \|_\infty &\leq \| w^{(N)}(0) \|_\infty + \sum_{j=1}^{\frac{T}{\Delta t}}  \Delta t L_c  \leq \| w^{(N)}(0) \|_\infty +  T~ L_c < \infty. \qedhere
\end{align*}
\end{proof} This statement seems to be sufficient to fulfill condition 6 in Theorem $\ref{Thm45}$. But since we are be interested in inverting the local densities later, we also require an additional bound from below. 

\begin{lemma}
\label{lemmaCarD}
Assume $c \in C^{0,1}(\mathbb{R})$ uniformly bounded, $v \in C^1(\mathbb{R}_+)$. Set $f(y,w) = c(y) v(w)$ and $v(w) = 1 - \frac{1}{w}$. Let $\varepsilon>0$ be used as in $(\ref{Bv-cond})$. If for the Lipschitz constant of $f$ in the second argument on $[1,\infty)$, $L_{f_w}$, the CFL condition $\frac{\Delta t}{\Delta y}L_{f_w}<1$ is satisfied, then a one-step Lax Friedrichs approximation from $(\ref{H})$ is bounded from below by $1+\tilde{\varepsilon}$ if
\begin{equation}
\label{carDistance}
\Delta t \leq \frac{\varepsilon - \tilde{\varepsilon}}{L_c}
\end{equation}
for $0<\tilde{\varepsilon}<\varepsilon$ and $L_c>0$ the Lipschitz constant of $c$. 
\end{lemma}

\begin{proof}
By $(\ref{Bv-cond})$ we know that $w_i^{(N)}(0)$ is bounded from below by $1+\varepsilon$. Denote $L_c>0$ the Lipschitz constant of $c$ and $L_v>0$ the Lipschitz constant of $v$. For $\lambda = \frac{\Delta t}{\Delta y}$ we get
\begin{align*}
w_i^{(N)} (t_1) =&  \frac{w^{(N)}_{i-1}(t_0) + w^{(N)}_{i+1}(t_0)}{2} + \frac{\lambda}{2} \Big(f(y_{i+\frac{3}{2}}, w^{(N)}_{i+1}(t_0)) - f(y_{i-\frac{1}{2}}, w^{(N)}_{i-1}(t_0))\Big)\\
\geq & \frac{w^{(N)}_{i-1}(t_0) + w^{(N)}_{i+1}(t_0)}{2} + \frac{\lambda}{2} \Big(c(y_{i+\frac{3}{2}})v(w_{i+1}^{(N)}(t_0)) - (c(y_{i+\frac{3}{2}}) + 2\Delta y L_c) v(w_{i-1}^{(N)}(t_0))\Big) \\
\geq & \frac{w^{(N)}_{i-1}(t_0) + w^{(N)}_{i+1}(t_0)}{2} - \frac{\lambda}{2} c(y_{i+\frac{3}{2}}) L_v \big|w_{i+1}^{(N)}(t_0) - w_{i-1}^{(N)}(t_0) \big| - \Delta t L_c v(w_{i-1}^{(N)}(t_0)).
\end{align*}
Using the CFL condition, the initial boundedness of $w_i^{(N)}(t_0)$ and the bound of the velocity function for an argument larger than 1, we get
\begin{align*}
w_i^{(N)} (t_1) &\geq \frac{1}{2} \big(2 \min \big\lbrace w_{i+1}^{(N)}(t_0), w_{i-1}^{(N)}(t_0) \big\rbrace \big) - \Delta t L_c v(w_{i-1}^{(N)}(t_0)) \\
& \geq 1 + \varepsilon - \Delta t L_c \geq 1+\tilde{\varepsilon}. \qedhere
\end{align*}
\end{proof}
The result is applicable for any other time step fulfilling the mentioned conditions. This lemma does not prove that the Lax-Friedrichs approximations can be uniformly bounded from below for any time horizon $T>0$. Let $(\Delta t)_n$ denote the step size and $\varepsilon_n$ the safety distance in the $n$-th step, where $\varepsilon_0=\varepsilon$. Set the partition 
\begin{align*}
\mathcal{P} = \Big\lbrace & k \in ((\Delta t)_n)_{n \in \lbrace1,...,N^k\rbrace} ~:~ (\Delta t)_n \leq \frac{\varepsilon^k_{n-1} - \varepsilon^k_n}{L_c} ~ \forall n=1,...,N^k,~\varepsilon^k_{N^k} \geq 0,  \\
&\varepsilon^k_n >0~\forall n=1,...,N^k-1,~ \varepsilon^k_{n+1} < \varepsilon^k_n~ \forall n=1,...,N^k-1, ~ N^k \in \mathbb{N} \Big\rbrace.
\end{align*}
We can only ensure the boundedness by 1 from below for 
\begin{equation}
\label{restrT}
T \leq \underset{k \in \mathcal{P}}{\sup} \Bigg(\sum_{i=1}^{N^k} (\Delta t)_{i}\Bigg) = \sum_{i=1}^{N^{k^*}} \frac{\varepsilon^{k^*}_{i-1} - \varepsilon^{k^*}_i}{L_c} = \frac{\varepsilon}{L_c}, 
\end{equation}
for ${k^*} =$ argmax $\Big(\underset{k \in \mathcal{P}}{\sup} \Big(\sum_{i=1}^{N^k} (\Delta t)_{i}\Big)\Big)$.\\
Now we apply Theorem $\ref{Thm45}$ to our setting with the inverse local densities.
\begin{theorem}
\label{corWeaksol}
Consider the Cauchy problem
\begin{align}
\label{csvlaw3}
\begin{split}
w_t - (\hat{c}(y,t) \tilde{v}(w))_y &= 0 ~~~ y\in [0,1], ~ t>0\\
w(0,y) &= w_0(y).
\end{split}
\end{align}
Assume that for the initial conditions it holds $w_0 \in L^\infty([0,1])$ and $(\ref{Bv-cond})$. Define $\hat{c} \in C^{0,1}([0,1] \times \mathbb{R}_+)$ as the smoothed version of the bounded capacity reduction function $\tilde{c}$, where at each discontinuity of $\tilde{c}$ we use a linear smoothing in an interval of length $\varepsilon>0$. The space dependent flux function is given by $f(\hat{c},w) = \hat{c} \tilde{v}(w)$, where $\tilde{v}(w) = 1 - \frac{1}{w} $. Using the Lagrangian grid, we set
\begin{equation*}
w_i^{(N)} (t_0) = \frac{1}{y_{i+\frac{1}{2}} - y_{i-\frac{3}{2}}} \int_{y_{i-\frac{3}{2}}}^{y_{i+\frac{1}{2}}} w_0^{(N)}(y) dy,
\end{equation*}
where $w_0^{(N)}$ was defined as in $(\ref{defw0})$.
Discretize the time domain $\mathbb{R}_+$ via $t_j = j \Delta t,~ j \in \mathbb{N}_0$ and the spatial component $[0,1]$ by using the Lagrangian grid and $y_{i-\frac{1}{2}} = \frac{i}{N},~ i \in \lbrace 0,1,...,N \rbrace$. Then $\Delta y = \frac{1}{N}$.
For $t_j>0$ define $w_i^{(N)} (t_j)$ as in $(\ref{H})$ and $(\ref{H2})$.
Assume that the CFL condition is met, i.e. $\frac{\Delta t}{\Delta y} \| f_u \|_\infty \leq 1$ and that condition $(\ref{carDistance})$ holds.
We define the piecewise constant local inverse density function for $i+j$ even as
\begin{equation*}
w_\Delta^{(N)}(y,t) = w_i^{(N)}(t_j) \text{,  for } (y,t) \in [y_{i-\frac{3}{2}}, y_{i+\frac{1}{2}}) \times [t_j, t_{j+1}).
\end{equation*} 
Passing if necessary to a subsequence, we have
\begin{equation*}
w_\Delta^{(N)} \rightarrow w
\end{equation*}
as $\Delta t,\Delta x \rightarrow 0$ and $N \rightarrow \infty$ in $L_{loc}^p ([0,1] \times \mathbb{R}_+)$ for any $p<\infty$ and where $w \in L^\infty([0,1] \times \mathbb{R}_+)$ is a weak solution to $(\ref{csvlaw3})$.
\end{theorem}
\begin{proof}
Equation $(\ref{H})$ defines the Lax Friedrichs scheme from Theorem $\ref{Thm45}$ for $(\ref{csvlaw3})$.
We show that the conditions of Theorem $\ref{Thm45}$ hold. 
Condition 1 is fulfilled by the boundedness of $w_0$ and $(\ref{Bv-cond})$. 
The function $\hat{c}$ is bounded per assumption and chosen to be Lipschitz continuous and thus, belongs to the functions of bounded variation (condition 2). 
The flux function is infinitely often differentiable in both arguments for $w>0$ such that condition 3 and 4 are fulfilled.
The second derivative with respect to the second argument is calculated by
$\frac{\partial^2}{\partial w^2} f(c,w) = - \frac{2c}{w^3} \neq 0$ 
for any $w \in [1,K], ~K >1$ and any strictly positive capacity functions and therefore condition 5 is met. 
Using Lemma $\ref{boundednessLemma}$ we obtain uniform bounds for the Lax Friedrichs approximation which directly gives condition 6. 
The CFL condition is forced to hold by an additional assumption in the theorem such that also condition 7 is satisfied. 
\end{proof}

\subsection{Convergence of the Local Density Functions}
After having found a convergent series of inverse local density functions to a limit $w$ and we have that $w$ is bounded from below by 1 due to Lemma $\ref{lemmaCarD}$ for sufficiently small time horizons $T>0$, we can conclude that for $t \in [0,T]$ there exists a limit function $\tilde{\rho}$ defined by
\begin{equation}
\label{rhoHut}
\hat{\rho}(y,t) = \frac{1}{w(y,t)}.
\end{equation}
This formulation still is given in Lagrangian coordinates. To convert the coordinates backwards and show that the limit function suits to the solution of the initial conservation law $(\ref{makroProblem})$, some computations are left to be examined.

One can show that the grid functions $x_i(t_j)$ can be used to construct a sequence of function that converges to a continuous function that returns the position of the vehicle with infinitesimal number $y \in [0,1]$ at time $t$ if $N L(N) = 1$.  Out of the initial vehicle position grid function, we define a piecewise linear function of Eulerian coordinates which we use for the limit process. For $y_{i-\frac{1}{2}} \in [0,1]$ set 
\begin{equation}
\label{constrX1}
\tilde{X}^{(N)} (y_{i-\frac{1}{2}},t_p) = \tilde{x}^{(N)}_i (t_p). 
\end{equation}
Between those grid points we choose a linear interpolation in both components by
\begin{align}
\begin{split}
\label{constrX2}
\tilde{X}^{(N)}(\cdot,t) =& \frac{1}{\Delta t} \Big( (t_{p+1} - t) \tilde{X}^{(N)}(\cdot,t_p) + (t-t_p) \tilde{X}^{(N)}(\cdot,t_{p+1}) \Big),~~t \in [t_p, t_{p+1}] \\ 
\tilde{X}^{(N)}(y,\cdot) =& \frac{1}{L(N)} \Big((y_{i+\frac{1}{2}} - y) \tilde{X}^{(N)}(y_{i-\frac{1}{2}}, \cdot) + (y - y_{i-\frac{1}{2}}) \tilde{X}^{(N)}(y_{i + \frac{1}{2}},\cdot) \Big),~y \in [y_{i-\frac{1}{2}}, y_{i+\frac{1}{2}}].
\end{split}
\end{align}

\begin{lemma}
\label{cvg_x}
There exists a continuous function $\hat{x}: [0,1] \times \mathbb{R}_+ \rightarrow \mathbb{R}$ such that for a subsequence of $(\tilde{X}^{(N)})_{N \in \mathbb{N}}$ defined as in $(\ref{constrX1})$ and $(\ref{constrX2})$, $\tilde{X}^{(N)} \rightarrow \hat{x}$ uniformly as $N \rightarrow \infty$. 
\end{lemma}

\begin{proof}
We know from $(\ref{Bv-cond})$ and Lemma $\ref{boundednessLemma}$ that all inverse local densities $w_i^{(N)}(t)$ are bounded for any $N \in \mathbb{N}$ and finite $t>0$ using the Lax-Friedrichs approximation from $(\ref{H})$. Therefore $\tilde{x}^{(N)}_i (t)$ are also bounded and due to its piecewise linear construction $\tilde{X}^{(N)}$ also stays bounded and is uniformly Lipschitz continuous in both arguments. This allows for the theorem of Arzela-Ascoli which ensures the convergence of a subsequence in $C([0,1] \times \mathbb{R}_+)$ for $N \rightarrow \infty$ against some limit function $\hat{x}$ which is also Lipschitz continuous. \qedhere
\end{proof}

As a next step, we try to show $L^1$-convergence of the constructed sequence of local densities against the function defined in $(\ref{rhoHut})$. For the vehicle positions we use the sequences from $(\ref{constrX1})$ and $(\ref{constrX2})$. 
To make notation a little easier, we assume in any step the rightmost vehicle to be located at the right border of the road interval $b$. In reality this might not always be true, but by a simple shift we obtain the actual location. Note that the leftmost vehicle will not be located at the leftmost point on the road for a finite number of vehicles. In the limit the vehicle with infinitesimal number 0 then is positioned at $x=a$.
\begin{align*}
\hat{x} (0,t) &= a, ~~~ \tilde{X}^{(N)}(1,t) = \hat{x} (1,t) = b,~~~ \forall N \in \mathbb{N}.
\end{align*}
Using Lemma $\ref{cvg_x}$ for a subsequence we get
\begin{equation}
\label{cvg_border}
\tilde{X}^{(N)} (0,t) \longrightarrow \hat{x}(0,t) = a
\end{equation}
uniformly as $N \rightarrow \infty$.\\
Note that $\tilde{X}^{(N)}$ can be inverted in $y$, since $\tilde{X}^{(N)}$ is constructed to be piecewise linear, continuous and strictly monotone increasing in $y$. The inverse will be denoted by $\tilde{Y}^{(N)} (x,t)$. 

Passing to a subsequence Lemma $\ref{cvg_x}$ showed the uniform convergence of $(\tilde{X}^{(N)}(t,\cdot))_{N \in \mathbb{N}}$. For fixed $t \in [0,T]$, the functions of the sequence were defined Lipschitz continuous and also injective since we assumed the indices of the cars to be ordered according to their position on the road. Thus, $\tilde{X}^{(N)}$ is strictly monotone in $y$. The same argument applies to the limit function. A result from analysis states that for a sequence of real, injective functions, which converges uniformly against an injective function $f$, also the sequence of the inverse functions converges uniformly to $f^{-1}$ (see e.g. \cite{InverseFolge}). This statement yields the uniform convergence of a subsequence of $(\tilde{Y}^{(N)}(\cdot ,t))_{N \in \mathbb{N}}$ to some function $\hat{y}(\cdot,t)$.

Now we collected all necessary results for the sequences for the position of the vehicles in both Eulerian and Lagrangian coordinates. Let us denote the density in Eulerian coordinates by
\begin{equation}
\label{rhoEN}
\tilde{\rho}^{E,(N)}(x,t) = \frac{1}{w^{(N)}(\tilde{Y}^{(N)}(x,t), t)}.
\end{equation}
In the Lagrangian coordinates we define
\begin{equation*}
\hat{\rho}^{(N)}(y,t) = \frac{1}{w^{(N)}(y,t)}.
\end{equation*}
As a last step, define 
\begin{equation*}
\tilde{\rho}^E (x,t) = \hat{\rho}(\hat{y}(x,t),t),
\end{equation*}
where $\hat{y}$ denotes the inverse of the limit function $\hat{x}$ with respect to the spatial variable.
This enables us to show $L^1([a,b])$-convergence of the Eulerian density function.

\begin{lemma}
For $\tilde{\rho}^{E,(N)}(\cdot,t)$ there exists a function $\tilde{\rho}^E(\cdot,t): [a,b] \rightarrow \mathbb{R}$ which is the limit of $\tilde{\rho}^{E,(N)}(\cdot,t)$ for $N \rightarrow \infty$ in $L^1([a,b])$.
\end{lemma}

\begin{proof}
\begin{align}
\| \tilde{\rho}^E (\cdot,t) - \tilde{\rho}^{E,(N)}(\cdot,t)\|_{L^1} &= \| \hat{\rho}(\hat{y}(\cdot,t),t) - \hat{\rho}^{(N)} (\tilde{Y}^{(N)}(\cdot,t),t) \|_{L^1}  \nonumber\\
& \leq \| \hat{\rho}(\hat{y}(\cdot,t),t) - \hat{\rho}(\tilde{Y}^{(N)}(\cdot,t),t) \|_{L^1} \nonumber\\
&~~ + \| \hat{\rho}(\tilde{Y}^{(N)}(\cdot,t),t) - \hat{\rho}^{(N)}(\tilde{Y}^{(N)}(\cdot,t),t) \|_{L^1}  \nonumber\\
& = \int_a^{\tilde{X}^{(N)}(0,t)} | \hat{\rho}(\hat{y}(t,x),t) - \hat{\rho}(\tilde{Y}^{(N)}(x,t),t) | dx \label{int1}\\
&~~+ \int_{\tilde{X}^{(N)}(0,t)}^b | \hat{\rho}(\hat{y}(x,t),t) - \hat{\rho}(\tilde{Y}^{(N)}(x,t),t) | dx \label{int2}\\
& ~~+ \int_a^{\tilde{X}^{(N)}(0,t)} | \hat{\rho}(\tilde{Y}^{(N)}(x,t),t) - \hat{\rho}^{(N)}(\tilde{Y}^{(N)}(x,t),t) | dx \label{int3}\\
&~~+ \int_{\tilde{X}^{(N)}(0,t)}^b | \hat{\rho}(\tilde{Y}^{(N)}(x,t),t) - \hat{\rho}^{(N)}(\tilde{Y}^{(N)}(x,t),t) | dx \label{int4}\\
&\longrightarrow 0, \nonumber
\end{align}
as $N \rightarrow \infty$. 

Both, $\hat{\rho}$ and $\hat{\rho}^{(N)}$ are bounded by 1. By $(\ref{cvg_border})$ the upper borders in $(\ref{int1})$ and $(\ref{int3})$ converge to the lower borders of the integral and thus the integrals go to 0 as $N \rightarrow \infty$. \\
In $(\ref{int2})$ the integrand is bounded by 2. Using Lemma $\ref{cvg_x}$, the arguments of $\hat{\rho}$ converge uniformly thus the integrand goes to zero almost everywhere. By the dominated convergence theorem the whole integral goes to 0. In $(\ref{int4})$ we used Theorem $\ref{Thm45}$ and the boundedness of the inverse local density function that gives us $L^1$-convergence of $\hat{\rho}^{(N)}$ to $\hat{\rho}$. Thus this integral also goes to 0 for $N \rightarrow \infty$. We showed that all four integrals tend to 0 for $N \rightarrow \infty$ which finishes the proof.
\end{proof}

\subsection{Equivalence of the Solutions of the Cauchy Problems}
\label{sec: EqSol}
So far we constructed a limit function of the local densities and we found a weak solution to the Cauchy problem of local inverse densities in $(\ref{cpIv})$. Now, Theorem $\ref{eqvSOL}$ claims a connection of the weak solutions of the Cauchy problems in $(\ref{makroProblem})$ and $(\ref{cpIv})$. 

\begin{theorem}[Equivalence of Weak Solutions]
\label{eqvSOL}
Let $\tilde{u} \in L^{\infty}(\mathbb{R} \times \mathbb{R}_+)$ be a weak solution to
\begin{equation*}
\Big(\frac{1}{\tilde{u}}\Big)_t - \big(\tilde{a}\tilde{v}(\tilde{u})\big)_y = 0,~~~~ \tilde{u}(y,0)=\tilde{u}_0(y).
\end{equation*}
Let the map $T$ be defined by
\[
      T: \left \lbrace \begin{array}{ll} \mathbb{R} \times \mathbb{R}_+ \rightarrow Im(y) \times \mathbb{R}_+ \\
          (x,t) \mapsto (y(x,t),\bar{t}(x,t)), \end{array}\right.
   \]
   where $y(x,t) = \int_{x(t)}^x \rho (z,t)dz$ and $\bar{t}(x,t) = t$.
 Let $x(t)$ be a vehicle trajectory satisfying $\frac{\partial}{\partial t}x(t) = a(x(t)) v(\rho(x(t),t))$. 
 Assume that $0 < \delta \leq \rho(x,t) \leq 1$ for all $(x,t) \in \mathbb{R} \times \mathbb{R}_+$ and that $v(\rho) > 0$ for $\rho > 0$ and $v \in L^{\infty}((\delta, 1])$ is Lipschitz continuous. Also assume $a \in L^{\infty}(\mathbb{R})$, being Lipschitz continuous and $a>0$.
Define $\tilde{a}=a \circ T^{-1}$, $\tilde{v} = v \circ T^{-1}$ and $\tilde{u_0} = u_0 \circ T^{-1}$.
Then $u= \tilde{u} \circ T$ is a weak solution to
\begin{equation*}
u_t + (a(x)v(u)u)_x = 0,~~~~ u(x,0) = u_0(x).
\end{equation*}
\end{theorem}

\begin{proof}
We consider a more general framework of a Cauchy problem in Lagrangian coordinates by $\tilde{D}_t - (\tilde{a}\tilde{G})_y = 0$ with initial condition $\tilde{D}(x,0)=\tilde{D}_0(x)$. 
Define $\tilde{D} = D \circ T^{-1}$ and $\tilde{G} = G \circ T^{-1}$. Consider a test function $\tilde{\phi} \in C_0^{0,1}(\mathbb{R} \times \mathbb{R}_+)$. Note that the convolution of such a test function with a mollifier $\lambda_\varepsilon$ creates a uniformly convergent sequence $\tilde{(\phi_\varepsilon)}_{\varepsilon>0}$ against a function in $C_0^\infty(\mathbb{R})$. 
Also the weak derivatives of the test function converge in $L^1(\mathbb{R})$ as $\varepsilon \rightarrow 0$ such that
\begin{align*}
&\int_0^\infty \int_\mathbb{R} \tilde{(\phi_\varepsilon)}_t \tilde{D} - \tilde{(\phi_\varepsilon)}_y(\tilde{a} \tilde{G}) dydt + \int_\mathbb{R}  \tilde{(\phi_\varepsilon)}(y,0) \tilde{D}_0(y)dy\\
&\overset{\varepsilon \rightarrow 0}{\longrightarrow} \int_0^\infty \int_\mathbb{R} \tilde{\phi}_t \tilde{D} - \tilde{\phi}_y \tilde{a} \tilde{G}dydt + \int_\mathbb{R}  \tilde{\phi}(y,0) \tilde{D}_0(y)dy.
\end{align*}
Thus, the weak formulation 
\begin{equation*}
\int_0^\infty \int_\mathbb{R} \tilde{\phi}_t \tilde{D} - \tilde{\phi}_y \tilde{a} \tilde{G}dydt + \int_\mathbb{R}  \tilde{\phi}(x,0) \tilde{D}_0(y)dy = 0
\end{equation*}
for a Cauchy problem also holds for Lipschitz continuous test functions $\tilde{\phi}$ with compact support.

Define $\tilde{\phi} = \phi \circ T^{-1}$. One can show that $T$ is a bi-Lipschitz homeomorphism and therefore $\phi$ is also a compact supported Lipschitz continuous test function on $\mathbb{R} \times \mathbb{R}_+$. Using a change of variables we get 
\begin{align*}
0 &= \int_0^\infty \int_\mathbb{R} \tilde{\phi}_t \tilde{D} - \tilde{\phi}_y \tilde{a} \tilde{G}dydt + \int_\mathbb{R}  \tilde{\phi}(y,0) \tilde{D}_0(y)dy\\
&= \int_0^\infty \int_\mathbb{R} \big((\tilde{\phi}_t \circ T) (\tilde{D} \circ T) - (\tilde{\phi}_y \circ T) ((\tilde{a} \tilde{G})\circ T)\big) \rho dxdt + \int_\mathbb{R}  \tilde{\phi}\circ(T(x,0)) \tilde{D}_0(T(y,0))dx\\
&= \int_0^\infty \int_\mathbb{R} \phi_t \rho D + \phi_x(\rho D av - aG) dx dt + \int_\mathbb{R}\phi(x,0) D_0(x) \rho(x) dx
\end{align*}
which is just the weak formulation for a solution of a Cauchy problem of the form
\begin{equation*}
(\rho D)_t + (\rho D av - aG)_x = 0, ~~~~ D(x,0) = D_0(x).
\end{equation*}
Inserting $D= \frac{1}{\rho}$ and $G=v$ leads to the trivial relation $1_t + 0_x = 0$. Using Lagrangian coordinates, the conservation law describes conservation of volume, where instead in Eulerian coordinates we have conservation of mass. As described in \cite{Wagner1}, to obtain conservation of mass out of conservation of volume we set $D=1$ and $G=0$. Then, we end up with the proposed relation
\begin{equation*}
\rho_t + (a(x)v(\rho) \rho)_x = 0,
\end{equation*}
where $\rho = \tilde{\rho} \circ T$. 
\end{proof}

All together we have shown the following theorem.
\begin{theorem}
Let $T>0$ and $a<b$. Consider the Cauchy problem 
\begin{align}
\begin{split}
\label{CSProb}
\rho_t + (F_{ac}(x,t,\rho))_x &= 0,~~~ x \in [a,b], ~ t \in [0,T]\\
\rho(x,0) &= \rho_0(x).
\end{split}
\end{align} 
Set $F_{ac}(x,t,\rho) = (c(x,t) f(\rho))$ with $f(\rho) = \rho v(\rho)$ and where the velocity functions is $v(\rho)=1-\rho$. Let $c: \mathbb{R} \times [0,T] \rightarrow [0,1]$ be a Lipschitz continuous and uniformly bounded capacity function. It is given by $c(x,t) = c_{ac}(x,t) c_{road} (x)$. Assume that there exist $0<\underline{c}, \bar{c}< \infty$ such that $c(x,t) \in [\underline{c}, \bar{c}]~~ \forall (x,t) \in [a,b] \times [0,T]$. 
Let the microscopic model be approximated by $(\ref{approxLxFstyle})$ and $(\ref{approxLxFstyle1})$. Assume that the initial vehicle positions are chosen such that $(\ref{Bv-cond})$ and $(\ref{carDistance})$ holds. For the time horizon $T$ we impose $(\ref{restrT})$.
Define $\tilde{\rho}^{E,(N)}$ as in $(\ref{rhoEN})$ and the sequence of inverse densities $((w_i^{(N)})_{i \in \lbrace 1,...,N \rbrace})_{N \in \mathbb{N}}$ as in $(\ref{H})$ and $(\ref{H2})$. Denote $L_{F_{ac}}$ the Lipschitz constant of the flux with respect to the second argument. Let $L(N) = \frac{1}{N}$ and assume that equidistant space and time grids are chosen such that 
\begin{equation*}
\frac{\Delta t}{\Delta x} L_{F_{ac}} < 1.
\end{equation*}
For $N \rightarrow \infty$ and $\Delta t, \Delta x \rightarrow 0$ $(\ref{rhoEN})$ converges in $C([0,T],L^1([a,b]))$ to a weak solution of the Cauchy problem $(\ref{CSProb})$.  
\end{theorem}
We finally showed that the local density functions from the microscopic model with macroscopic accidents can be used to construct a convergent series against a weak solution of $(\ref{CSProb})$ for increasing the total number of vehicles to $\infty$ and keeping total mass constant. 

\section{Numerical Results}
\label{numRes}
In this section we investigate the convergence results from a numerical point of view and use simulations to show that also the local densities from the microscopic model converge to the traffic density of the corresponding macroscopic model.

\subsection{Numerical Treatment of the Microscopic Model}
Both traffic accident models are characterized by a deterministic evolution that is interrupted by stochastic jumps being either new accidents or removals of an accident. For step sizes $0<\Delta t \leq \frac{L}{v_{max}}$ between these jumps due to \begin{alignat*}{1}
&P\Big(x_i^{n+1} = x_i^n + \Delta t c_{road}(x_i^n)  c_{ac}(x_i^n)  v \bigg( \frac{L}{\Delta x_i^n} \bigg)\mid X^n\Big) = 1,~~ i=1,...,N-1 \label{x_1} 
\end{alignat*}
the vehicle positions can directly computed by 
\begin{equation*}
    x_i^{n+1} = x_i^n + \Delta t c_{road}(x_i^n)  c_{ac}(x_i^n)  v \Big( \frac{L}{\Delta x_i^n}\Big)
\end{equation*}
for vehicle $i=1,...,N-1$. For the last vehicle we set
\begin{equation*}
    x_N^{n+1} = x_N^n + \Delta t c_{road}(x_N^n)  c_{ac}(x_N^n)  v \Big( \frac{L}{x_{1}^n-x_N^n + (b-a)}\Big)
\end{equation*}
capturing periodic boundary conditions on a road $[a,b]$.
To avoid numerical difficulties the piecewise constant capacity function, which may have up to finitely many discontinuities, has to be smoothed. Simulations show that a linear smoothing is sufficient here. 
Since the accident components stay constant between the jumps we denote by
\begin{equation*}
    \phi_{\Delta t}(x^n,M^n, p^n, s^n,c^n,u^n,l^n) =  (x^{n+1},M^n,p^n,s^n,c^n,u^n,l^n)
\end{equation*}
the deterministic evolution of the stochastic process. 

For the simulation in chronological order we proceed as follows. After initializing the system in any time step using
\begin{alignat*}{2}
&P(u^{n+1} = 0 \mid X^n) &&= 1 - \Delta t \psi(X^n), \\
&P(u^{n+1} = 1 \mid X^n) &&=  \Delta t \psi(X^n) \frac{\lambda^A}{\lambda^A + M^n \lambda^R},~~~P(u^{n+1} = -1 \mid X^n) =  \Delta t \psi(X^n) \frac{M^n \lambda^R}{\lambda^A + M^n \lambda^R}
\end{alignat*}
the presence and kind of an event can be obtained by sampling two independently generated Bernoulli random variables. The first one chooses a new event with probability $p^n=\Delta t \psi(X^n)$ and the second one chooses a new accident with probability $q^n=\frac{\lambda^A}{\lambda^A + M^n \lambda^R}$. In case there is no new event we use $\phi_{\Delta t}$ to compute the deterministic evolution of the system. In case of a new accident another Bernoulli random variable with parameter $\beta$ is used to determine whether an accident of type 1 or type 2 is observed.

For a new accident we use the inverse transformation method to sample from the respective probability measure for the accident position. The measure for the accident positions due to high flux can be directly be computed using the piecewise constant shape of $h_{ac}$
\begin{equation*}
\mu_{t,ac}^F (B) = \frac{\int_B h_{ac}(x,t) dx}{C_F(t,ac)} = \ \frac{\sum_{i=1}^{N-1}\varepsilon_{x_i(t)}(B) c(x_i(t) \rho(x_i(t)) (1-\rho(x_i(t))) (x_{i+1}(t) - x_i(t))}{\sum_{i=1}^{N-1}c(x_i(t) \rho(x_i(t)) (1-\rho(x_i(t))) (x_{i+1}(t) - x_i(t))}
\end{equation*}
for $B \in \mathcal{B}([a,b])$. With adequate methods we sample accident sizes from $\mu^s$ and accident capacity reductions from $\mu^{cap}$. For the removal of an accident we uniformly choose one of the active accidents using a discrete inverse transformation method.
The algorithm is repeated in the next time step until the time horizon is reached.

\subsection{Numerical Treatment of the Macroscopic Model}
A Lax-Friedrichs scheme is used to solve the macroscopic model. For a better comparison with the microscopic model we assume a bounded road $[a,b]$ with periodic boundary conditions. We use step sizes $\Delta t, \Delta x >0$ such that the CFL condition is satisfied and $K \Delta x = b-a$ for some $K \in \mathbb{N}$. We set $x_{\frac{1}{2}} = a$, $x_{K+\frac{1}{2}}=b$ and define by 
\begin{equation*}
    \rho_i^0 = \frac{1}{\Delta x} \int_{x_{i-\frac{1}{2}}}^{x_{i+\frac{1}{2}}} \rho_0(x) dx
\end{equation*}
the initial cell means for $x_i = a - \frac{\Delta x}{2} + i \Delta x,~i \in \lbrace 1,...,K \rbrace$. For the flux function $f(c,\rho) = c \rho (1-\rho)$ the steps for the Lax-Friedrichs scheme are given by
\begin{align*}
    \rho_i^{j+1} =&~ \frac{\rho_{i+1}^j + \rho_{i-1}^j}{2} - \frac{\Delta t}{2\Delta x} \Big(f(c_{i+1}^j, \rho_{i+1}^j) - f(c_{i-1}^j, \rho_{i-1}^j) \Big),
\end{align*}
where $c_i^j = c(x_i,t_j), ~i  \in \lbrace 2,...,K-1 \rbrace$, $j \in \mathbb{N}_0$. For $\rho_1^{j+1}$ and $\rho_K^{j+1}$ we have to adapt for the periodic boundary conditions
\begin{align*}
    \rho_1^{j+1} =&~ \frac{\rho_{2}^j + \rho_{K}^j}{2} - \frac{\Delta t}{2\Delta x} \Big(f(c_{2}^j, \rho_{2}^j) - f(c_{K}^j, \rho_{K}^j) \Big) \\
    \rho_K^{j+1} =&~ \frac{\rho_{1}^j + \rho_{K-1}^j}{2} - \frac{\Delta t}{2\Delta x} \Big(f(c_{1}^j, \rho_{1}^j) - f(c_{K-1}^j, \rho_{K-1}^j) \Big).
\end{align*}
The probability measure for the accident positions $\mu_{t_j}^{F_{ac},Mac}$ for $B \in \mathcal{B}([a,b])$ is approximated by
\begin{equation*}
\mu_{t_j}^{F_{ac},Mac}(B)=\frac{\sum_{i=1}^K c_i^j \rho_i^j (1-\rho_i^j) \int_B \mathbbm{1}_{[x_{i-\frac{1}{2}},x_{i+\frac{1}{2}})}(x) dx}{\sum_{i=1}^K c_i^j \rho_i^j (1-\rho_i^j) \Delta x}.
\end{equation*}
The discretization of $\mu_{t_j}^{D,Mac}$ can be computed using the piecewise constant density segments from the Lax-Friedrichs scheme
\begin{equation*}
    \mu_{t_j}^{D,Mac}(B) = \frac{\sum_{i=2}^K \varepsilon_{x_{i-\frac{1}{2}}}(B) (\rho_{i}^j - \rho_{i-1}^j)_+ + \varepsilon_{x_{\frac{1}{2}}}(B) (\rho_{1}^j - \rho_{K}^j)_+ }{\sum_{i=2}^K (\rho_{i}^j - \rho_{i-1}^j)_+ + (\rho_{1}^j - \rho_{K}^j)_+}.
\end{equation*}
Apart from the mentioned adaptions one can execute the simulation quite similarly to the microscopic one. For further details we refer to \cite{A}.

\subsection{Exemplary Numerical Comparison}
For a microscopic example we assume a finite road represented by the interval $[-10,10]$ with periodic boundary conditions such that any vehicle leaving at $x=10$ just enter again at $x=-10$. On the road, we set $N=1600$ vehicles, each initially having the same distance to the front vehicle. The length of a vehicle is chosen to be $L(N)=\frac{8}{N}$.

The capacity function of the road is assumed to be $c_{road}(x) = 7 - 2 \cdot \mathbbm{1}_{[0,5]} (x)$. We can interpret this as a general speed limit of 70 kilometers per hour. An exception is made on the part of the interval $[0,5]$ where we allow only for 50 kilometers per hour. The capacity function is smoothed at each discontinuity by a linear interpolation on an interval of length $\varepsilon = 0.02$.  Accidents are incorporated by the accident capacity function $c_{ac}(x)=\prod_{j=1}^{M} 1-c_j \mathbbm{1}_{\big[p_j - \frac{s_j}{2}, p_j + \frac{s_j}{2}\big]} (x)$. Furthermore we assume $\Delta x = \frac{1}{80}$, $\Delta t = \frac{\Delta x}{10}$ and a time horizon of $T=10$. 

For the accidents we set $\lambda^F = \frac{1}{160}$, $\lambda^D = \frac{1}{50}$ and $\lambda^R = 0.25$ and $\beta = 0.5$, such that we allow for both types of accidents with the same likeliness. The distribution for the accident sizes is chosen to be a uniform distribution on $[0.2,1]$ and the capacity reduction is chosen to be distributed according to $ 0.5 \cdot \varepsilon_{0.5} + 0.5 \cdot \varepsilon_{0.99}$. Additionally, for the macroscopic model the flux function is chosen to be $f(\rho) = \rho(1- \rho)$ and the initial density to be $\rho_0(x) = 0.4$.

For one realization, we compare the local densities from the microscopic model with the densities from the macroscopic model on the left and the probability distribution for the accident location on the right of the following figures. For the second one we divide the road into 10 segments of length 2. The left bars represent the microscopic model and the right bars the macroscopic one. The lower parts of the bars show the share of the likeliness coming from accident type 1, whereas the upper parts show the share of accidents of type 2. 

\begin{figure}[htb!]
\centering
\subfloat[][]{\label{fig8a}\includegraphics[scale=0.492]{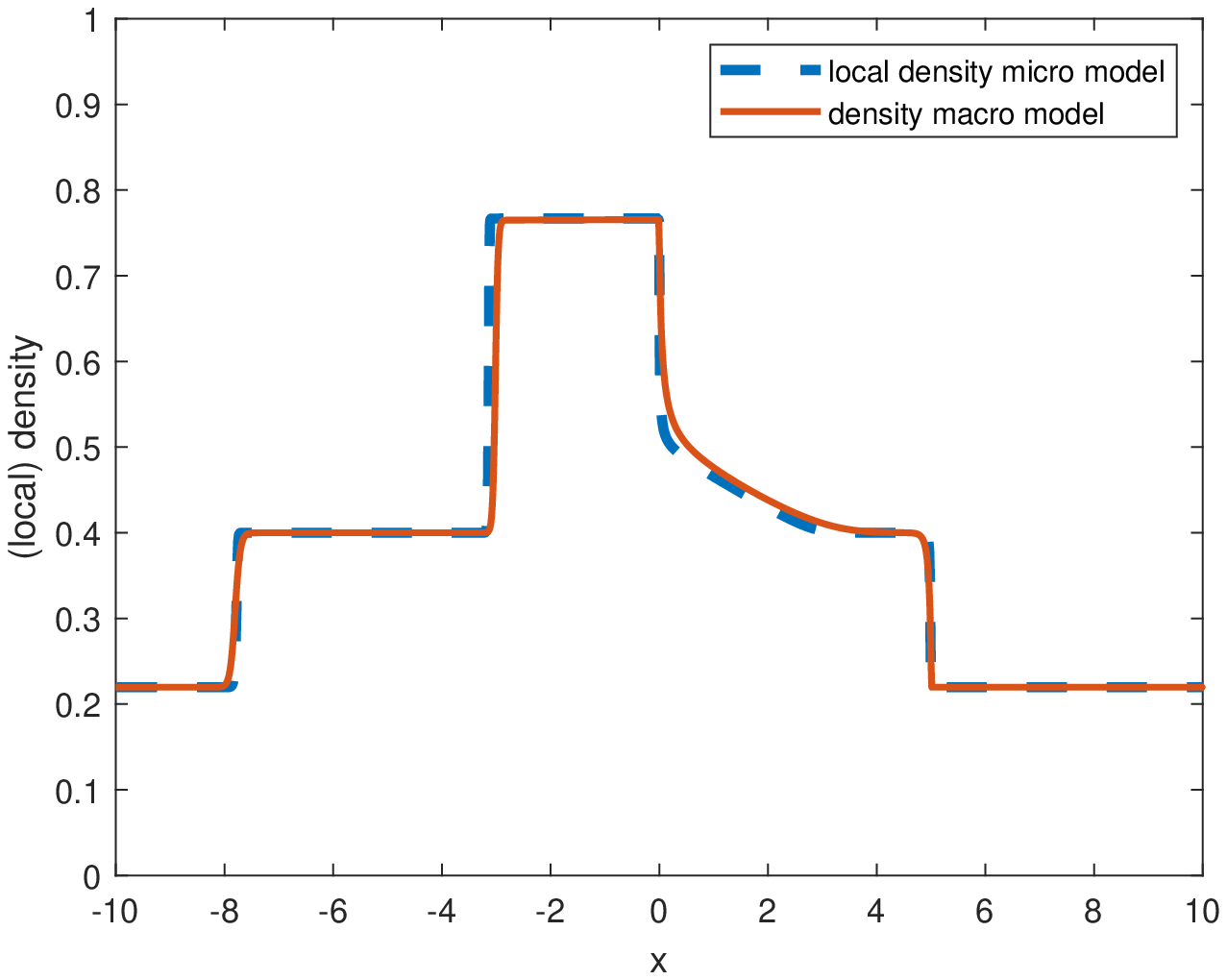} }
\qquad
\subfloat[][]{\label{fig8b}\includegraphics[scale=0.492]{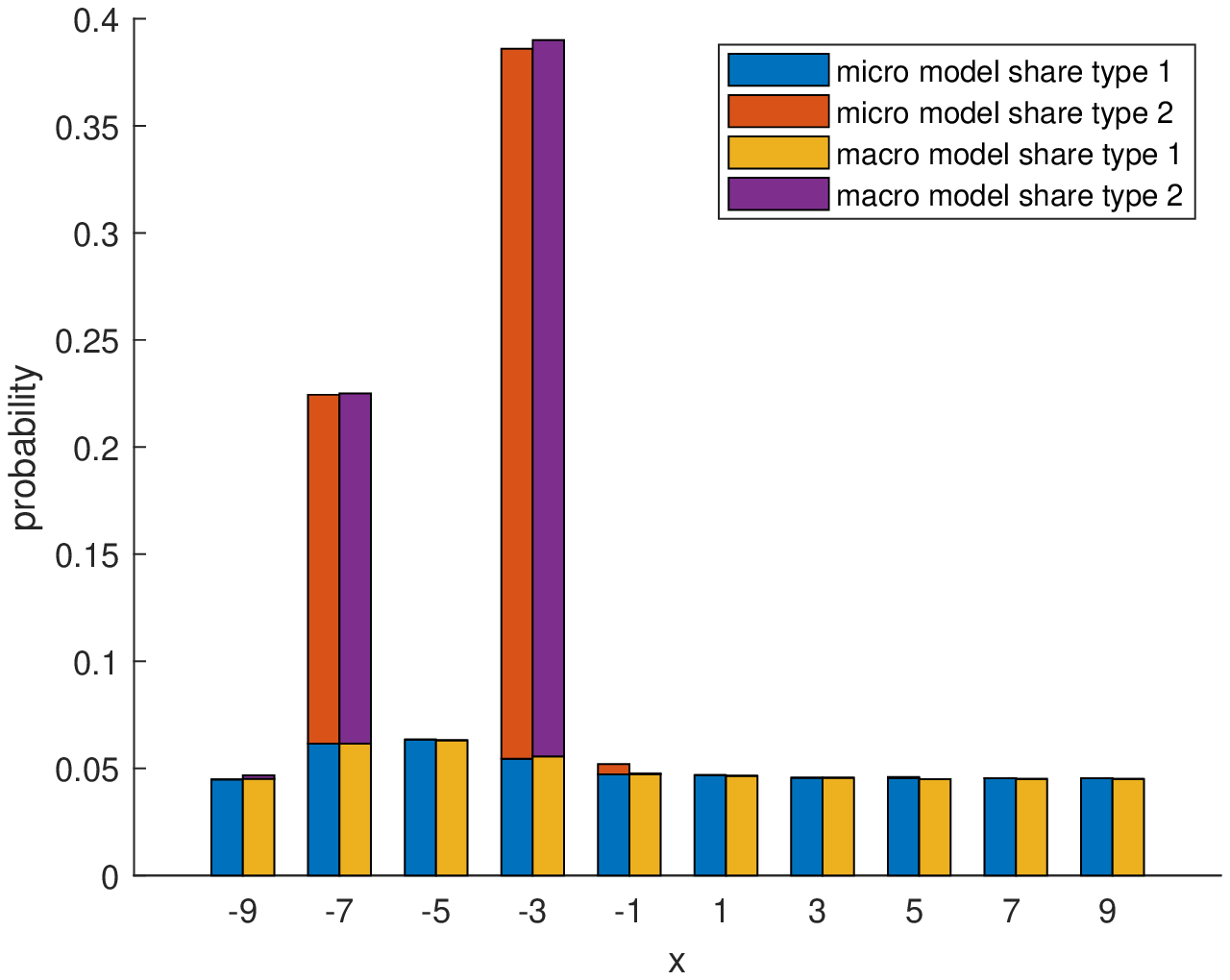}}
\caption{Local traffic density from the microscopic model and density function from the macroscopic model (a) and probability distribution of the accident location for both models (b) at $t=2.72$.}
\end{figure}

The first event is chosen to be at $t=2.72$. At that time Figure $\ref{fig8a}$ shows that both density functions coincide almost everywhere.

The likeliness of an accident of type 1, given by the lower bars in Figure $\ref{fig8b}$, seems to be distributed quite uniformly over the road in both models. This is due to quite stable and moderate traffic densities and velocities. The upper segments in this chart show the probability contribution of accidents of type 2. In contrast to type 1 accidents they are far from being approximately uniformly distributed, since they depend  on the positive increases in the local density functions. There are basically two sections where we observe an increase in the local density function. The first in the area of the interval $[-4,-2]$ right in front of the more restrictive speed limit and the second one being around $x=-8$ where the local densities recover to their initial level after the speed limit constraint. Therefore, we find large upper bars especially in those two sections in Figure $\ref{fig8b}$. Comparing both models, also the accident position probabilities look very similar. Indeed, in our particular example, the first accident is one of type 2 and happens at $x=-3.05$ in the macroscopic model and at $x=-3.13$ in the microscopic model. 

\begin{figure}[htb!]
\centering
\subfloat[][]{\label{fig9a}\includegraphics[scale=0.492]{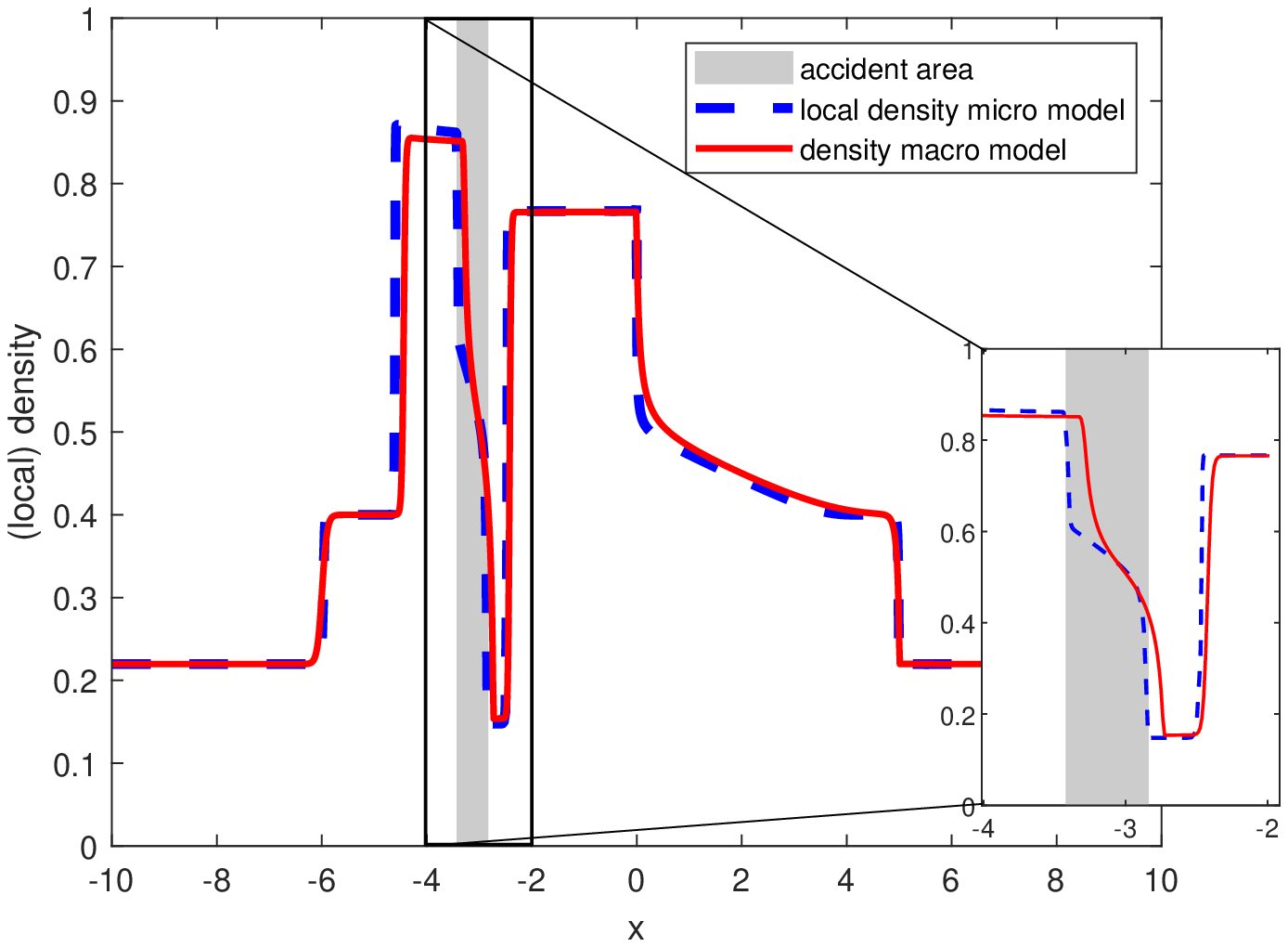}}
\qquad
\subfloat[][]{\label{fig9b}\includegraphics[scale=0.492]{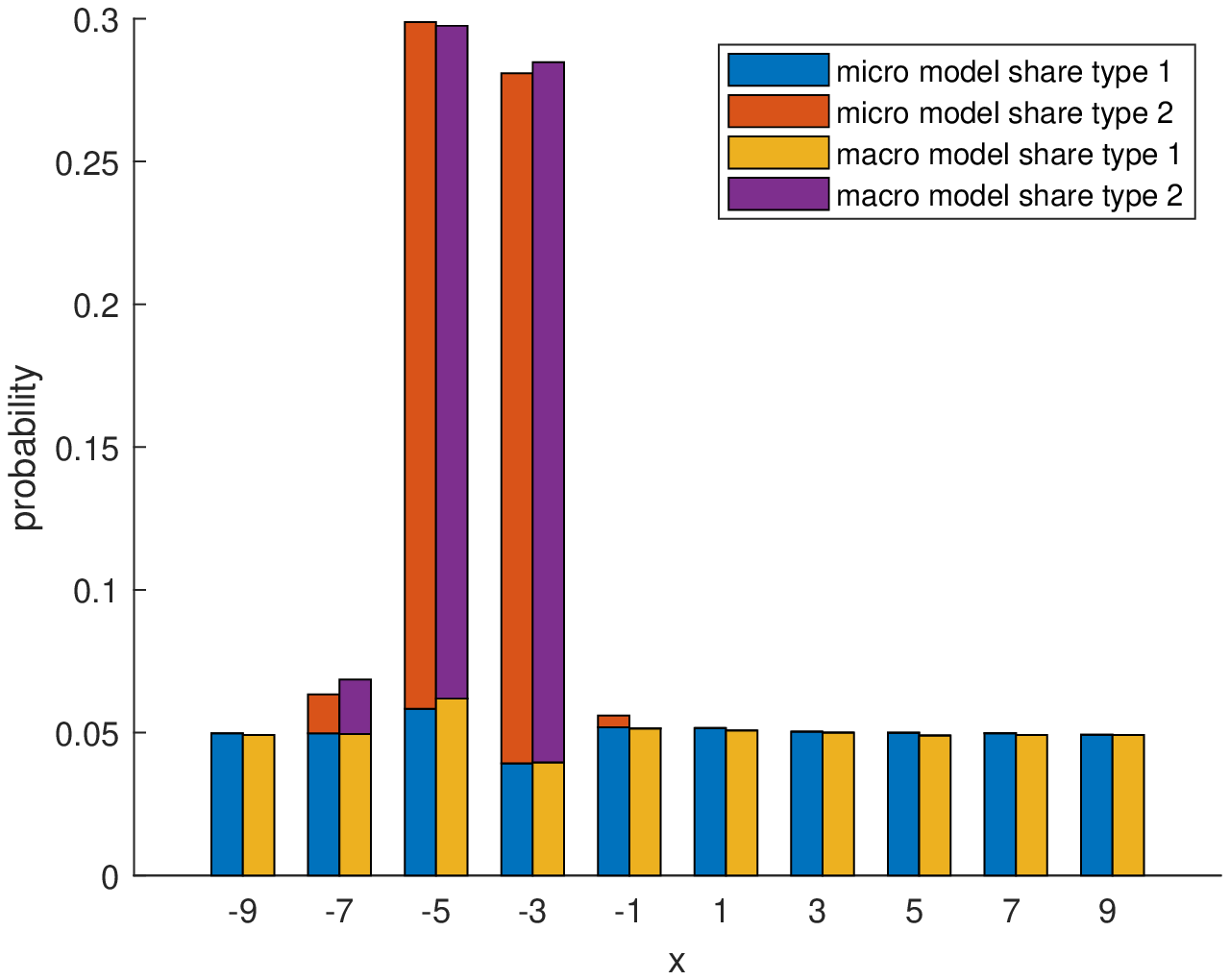}}
\caption{Local traffic density from the microscopic model and density function from the macroscopic model (a) and probability distribution of the accident location for both models (b) at $t=3.4$.}
\label{fig9}
\end{figure}

In both models at time $t=3.4$ the next event happens which is another accident in this example and we can already analyze the consequences of the first accident in Figure $\ref{fig9}$. We observe a steep increase in both densities right in front of the location of the accident in Figure $\ref{fig9a}$ and also a drop right after the accident. The grey area represents the segment of the first accident in the microscopic model. The second increase is due to the more restrictive speed limit on the interval $[0,5]$. Again both densities coincide well.

We observe quite uniform likeliness for an accident of type 1 in Figure $\ref{fig9b}$ apart from the segment around $x=-3$. There, first the density is close to one, which means that vehicles have a velocity of almost 0 and afterwards there is a very restrictive capacity reduction due to the accident. In both cases the flux is very small and leads to a low contribution for high-flux accidents around $x=-3$. Two sections where we observe an increase in the local density function are recognizable around $x=-5$ and $x=-2$. Apart from some minor differences for accidents of type 2 the accident position probabilities are similar in both models. 

In our case an accident of type 1 happens at $x=4.04$ in the microscopic and at $x=4.09$ in the macroscopic model. 

\subsection{Numerical Convergence Analysis}
In the microscopic model we denote $\rho_i^{(N)}(t)$ as the local density of vehicle $i$ in an environment with $N \in \mathbb{N}$ cars on the road at time $t>0$ defined as in $(\ref{odeAcc})$. We consider the piecewise constant local density function defined by
\begin{equation*}
\rho^{(N),Mic}(x,t) = \rho_i^{(N)}(t) \text{,    for  } x \in [x_i(t), x_{i+1}(t)), ~~~i=1,...,N-1.
\end{equation*}
In the periodic environment of our simulation we have to take care about the beginning and the end separately.
\begin{equation*}
\rho^{(N),Mic}(x,t) = \rho_N^{(N)}(t) \text{,    for  } x \in [a, x_{1}(t)) \cup [x_{N}(t),b]. 
\end{equation*} 
In the limit of $N \rightarrow \infty$ we expect that $\rho^{(N),Mic}$ converges to the density function $\rho^{Mac}$ of the macroscopic model, which is the weak solution to the Cauchy problem in $(\ref{makroProblem})$.

One error measure that could be considered is the expected value of the $L^1$-error of the density function from the macroscopic model and a piecewise constant version of the local density functions from the microscopic model, if we increase the number of vehicles in the microscopic model keeping total mass constant.
\begin{align*}
Err_{1} &= \mathbb{E}\bigg[\int_{a}^b | \rho^{(N),Mic}(x,t) - \rho^{Mac}(x,t) | dx \bigg] \\
& := \mathbb{E}\bigg[ \Delta x \sum_{i=0}^{\frac{b-a}{\Delta x}} | \rho^{(N),Mic}(a + i \Delta x,t) - \rho^{Mac}(a + i \Delta x,t) |  \bigg].
\end{align*}
We consider the density functions on an equidistant space grid and then use the rectangular rule to calculate the integral in all segments for the spatial component. A Monte Carlo simulation is performed to approximate the expected value. 

Both models should be related in some stochastic way. More precisely in any decision in which randomness plays a role, we have to make sure to use the same element $\omega \in \Omega$, where $(\Omega, \mathcal{A}, P)$ is the underlying probability space.

We also investigate how the third model introduced in Section $\ref{TildeModel}$ behaves. 
Defining the local density of the microscopic model with macroscopic accidents as $\tilde{\rho}_i^{(N)}(t)$ as in $(\ref{locDensityTilde})$ we can again define a piecewise constant local density function for $t>0$ by
\begin{align*}
\tilde{\rho}^{(N),Mic}(x,t) &= \tilde{\rho}_i^{(N)}(t) \text{,    for  } x \in [\tilde{x}_i(t), \tilde{x}_{i+1}(t)), ~~~i=1,...,N-1. \\
\tilde{\rho}^{(N),Mic}(x,t) &= \tilde{\rho}_N^{(N)}(t) \text{,    for  } x \in [a, \tilde{x}_{1}(t)) \cup [\tilde{x}_{N}(t),b]. 
\end{align*}
Comparing the microscopic model with accidents according to the macroscopic one to the macroscopic one for fixed $t>0$ we are interested in the expected $L^1([a,b])$-error given by
\begin{align*}
\begin{split}
\label{error2}
Err_{2} &= \mathbb{E}\bigg[\int_{a}^b | \tilde{\rho}^{(N),Mic}(x,t) - \rho^{Mac}(x,t) | dx \bigg] \\
& := \mathbb{E}\bigg[ \Delta x \sum_{i=0}^{\frac{b-a}{\Delta x}} | \tilde{\rho}^{(N),Mic}(a + i \Delta x,t) - \rho^{Mac}(a + i \Delta x,t) |  \bigg].
\end{split}
\end{align*}
Instead of using the expected value in the error measure we can also use $L^2(\Omega)$ norm which in the pure microscopic case is discretized by
\begin{equation*}
Err_{3} = \mathbb{E}\bigg[ \bigg(\Delta x \sum_{i=0}^{\frac{b-a}{\Delta x}} | \rho^{(N),Mic}(a + i \Delta x,t) - \rho^{Mac}(a + i \Delta x,t) | \bigg)^2  \bigg]^{\frac{1}{2}}.
\end{equation*}
Replacing the expected value by the $L^2(\Omega)$ norm for the model from Section $\ref{TildeModel}$ we obtain another error measure
\begin{equation*}
    Err_4 = \mathbb{E}\bigg[ \bigg( \Delta x \sum_{i=0}^{\frac{b-a}{\Delta x}} | \tilde{\rho}^{(N),Mic}(a + i \Delta x,t) - \rho^{Mac}(a + i \Delta x,t) | \bigg)^2  \bigg]^\frac{1}{2}.
\end{equation*}
To evaluate the errors numerically we performed a Monte Carlo simulation with 600 runs. If not mentioned differently, we choose the parameters as before.
Set the step size of the macroscopic space grid to $\Delta x = \frac{1}{160}$ and the length of one step in the time grid $\Delta t = \frac{\Delta x}{10}$ such that the CFL condition is satisfied. 

First, we compare the behaviour of the errors for different numbers of vehicles in the microscopic setting in Table $\ref{tab:1}$ choosing $N=\lbrace50,100,200,400,800,1600,3200\rbrace$. Due to the high diffusion of the Lax-Friedrichs approximation one notices numerically that in the limit it is not as precise as the theory suggests. After a decrease in all four error measures we are left with some kind of basic error.  

\begin{table}[htb!]
\centering
\begin{tabular}{r|c|c|c|c|c|c|c}

\hline
 & $N=50$  & $N=100$ & $N=200$ & $N=400$ & $N=800$ & $N=1600$ & $N=3200$ \\
\hline
$Err_1$ & $1.3712$ & $0.7844$ & $0.4185$ & $0.3134$ & $0.3287$ & $0.3914$ & $0.4355$\\
\hline
$Err_2$ & $0.9354$ & $0.5066$ & $0.2560$ & $0.1563$ & $0.1421$ & $0.1589$ & $0.1719$\\
\hline
$Err_3$ & $2.0403$ & $1.4746$ & $0.9149$ & $0.8764$ & $0.8150$ & $0.8357$ & $0.9116$\\
\hline
$Err_4$ & $1.1609$ & $0.6428$ & $0.3121$ & $0.2090$ & $0.2190$ & $0.2435$ & $0.2590$\\
\hline
\end{tabular}
\caption{Error measures for varying $N$ and for $\Delta x = \frac{1}{160}$ using the Lax-Friedrichs scheme.}
\label{tab:1}
\end{table}
To increase the accuracy of the limit investigation we suggest to use a space dependent Godunov scheme for the macroscopic model and define by 
\begin{equation*}
    \rho_i^0 = \frac{1}{\Delta x} \int_{x_{i-\frac{1}{2}}}^{x_{i+\frac{1}{2}}} \rho_0(x) dx
\end{equation*}
the initial cell means for $x_i = a - \frac{\Delta x}{2} + i \Delta x,~i \in \lbrace 1,...,K \rbrace$. For the flux function $f(c,\rho) = c \rho (1-\rho)$ we observe the maximum flux at $\rho^* = \frac{1}{2}$. The Godunov steps are given by
\begin{align*}
    \rho_i^{j+1} =&~ \rho_i^j - \frac{\Delta t}{\Delta x} \Big(\min \big\lbrace f(c_{i+1}^j\max \lbrace \rho_{i+1}^j, \rho^* \rbrace), f(c_{i}^j\min \lbrace \rho_{i}^j, \rho^* \rbrace)  \big\rbrace  \\
    &- \min \big\lbrace f(c_{i}^j\max \lbrace \rho_{i}^j, \rho^* \rbrace), f(c_{i-1}^j\min \lbrace \rho_{i-1}^j, \rho^* \rbrace)  \big\rbrace  \Big),
\end{align*}
where $c_i^j = c(x_i,t_j), ~i  \in \lbrace 2,...,K-1 \rbrace$, $j \in \mathbb{N}_0$. For $\rho_1^{j+1}$ and $\rho_K^{j+1}$ we have to adapt for the periodic boundary conditions
\begin{align*}
    \rho_1^{j+1} =&~ \rho_1^j - \frac{\Delta t}{\Delta x} \Big(\min \big\lbrace f(c_{2}^j\max \lbrace \rho_{2}^j, \rho^* \rbrace), f(c_{1}^j\min \lbrace \rho_{1}^j, \rho^* \rbrace)  \big\rbrace  \\
    &- \min \big\lbrace f(c_{1}^j\max \lbrace \rho_{1}^j, \rho^* \rbrace), f(c_{K}^j\min \lbrace \rho_{K}^j, \rho^* \rbrace)  \big\rbrace  \Big), \\
    \rho_K^{j+1} =&~ \rho_K^j - \frac{\Delta t}{\Delta x} \Big(\min \big\lbrace f(c_{1}^j\max \lbrace \rho_{1}^j, \rho^* \rbrace), f(c_{K}^j\min \lbrace \rho_{K}^j, \rho^* \rbrace)  \big\rbrace  \\
    &- \min \big\lbrace f(c_{K}^j\max \lbrace \rho_{K}^j, \rho^* \rbrace), f(c_{K-1}^j\min \lbrace \rho_{K-1}^j, \rho^* \rbrace)  \big\rbrace  \Big).
\end{align*}
Table $\ref{tab:1God}$ shows the error evolution if we use the Godunov scheme for the limit consideration. In all four cases we observe a decrease in the errors for increasing number of vehicles. The results support the idea of the model introduced in Section $\ref{TildeModel}$ being in between the microscopic and macroscopic one because $Err_2$ and $Err_4$ seem to be sub errors of $Err_1$ and $Err_3$, respectively.
\begin{table}[htb!]
\centering
\begin{tabular}{r|c|c|c|c|c|c|c}

\hline
 & $N=50$  & $N=100$ & $N=200$ & $N=400$ & $N=800$ & $N=1600$ & $N=3200$ \\
\hline
$Err_1$ & $1.5952$ & $1.0549$ & $0.5405$ & $0.3168$ & $0.1836$ & $0.1087$ & $0.0453$\\
\hline
$Err_2$ & $1.0357$ & $0.6265$ & $0.3335$ & $0.1831$ & $0.1013$ & $0.0571$ & $0.0320$\\
\hline
$Err_3$ & $2.3392$ & $1.8600$ & $1.0074$ & $0.7483$ & $0.6217$ & $0.4459$ & $0.1040$\\
\hline
$Err_4$ & $1.3000$ & $0.8149$ & $0.4115$ & $0.2110$ & $0.1139$ & $0.0637$ & $0.0358$\\
\hline
\end{tabular}
\caption{Error measures for varying $N$ and for $\Delta x = \frac{1}{160}$ using the Godunov-scheme.}
\label{tab:1God}
\end{table}

We also consider the evolution of the logarithms of the errors in time for a fixed number of vehicles $N=3200$ in the microscopic model in Figure $\ref{fig10}$. We observe that $Err_2$ and $Err_4$ stay more or less constant after they reached a certain level. The same applies for $Err_1$ whereas $Err_3$ increases more steeply and even shows a slight decrease after some time but still remains on a higher level than the other error measures. These results suggest that the error measures are not expected to show a significant increase if one considers larger time horizons.   
 
\begin{figure}[htb!]
\centering
\subfloat[][]{\label{fig10a}\includegraphics[scale=0.492]{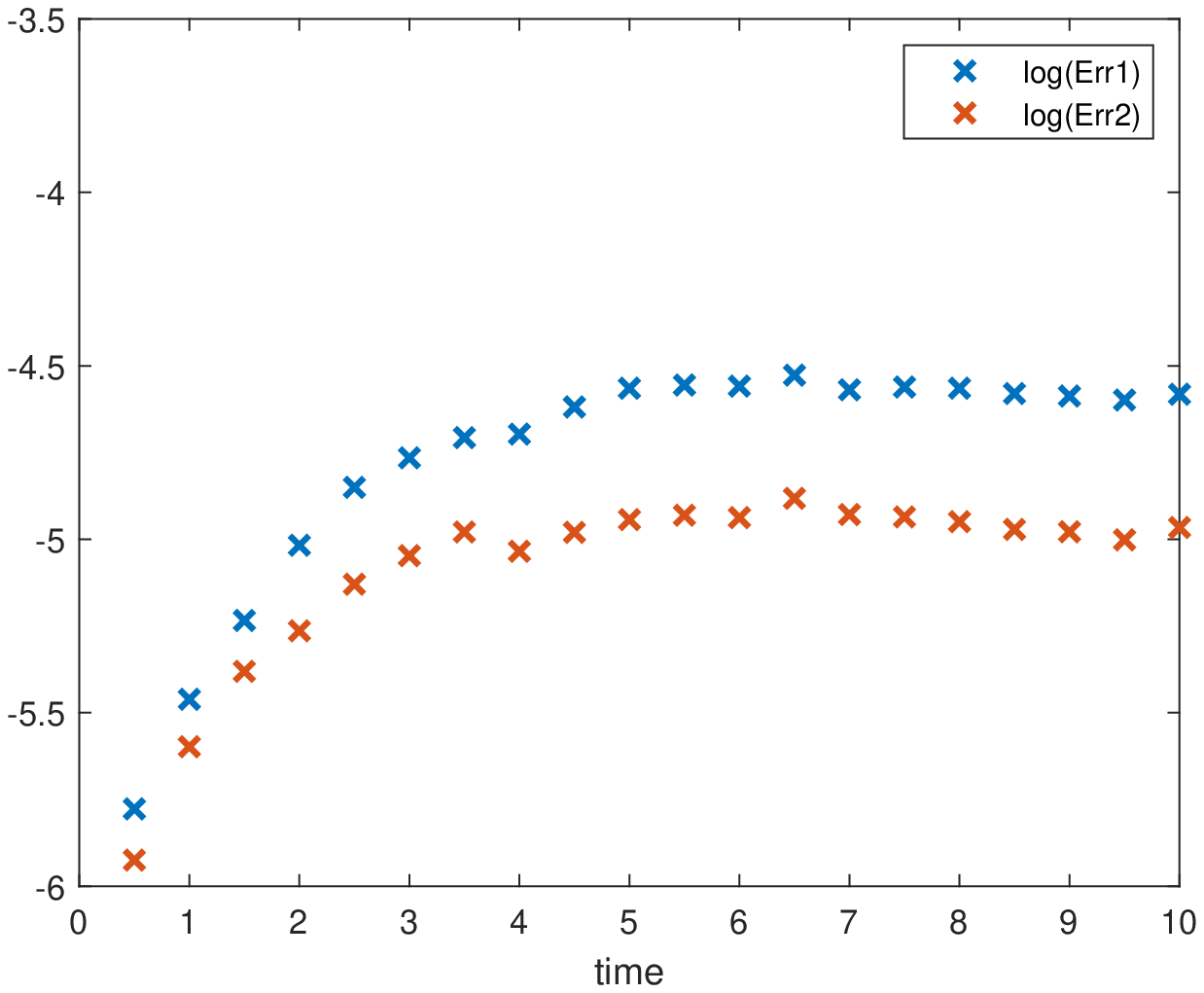}}
\qquad
\subfloat[][]{\label{fig10b}\includegraphics[scale=0.492]{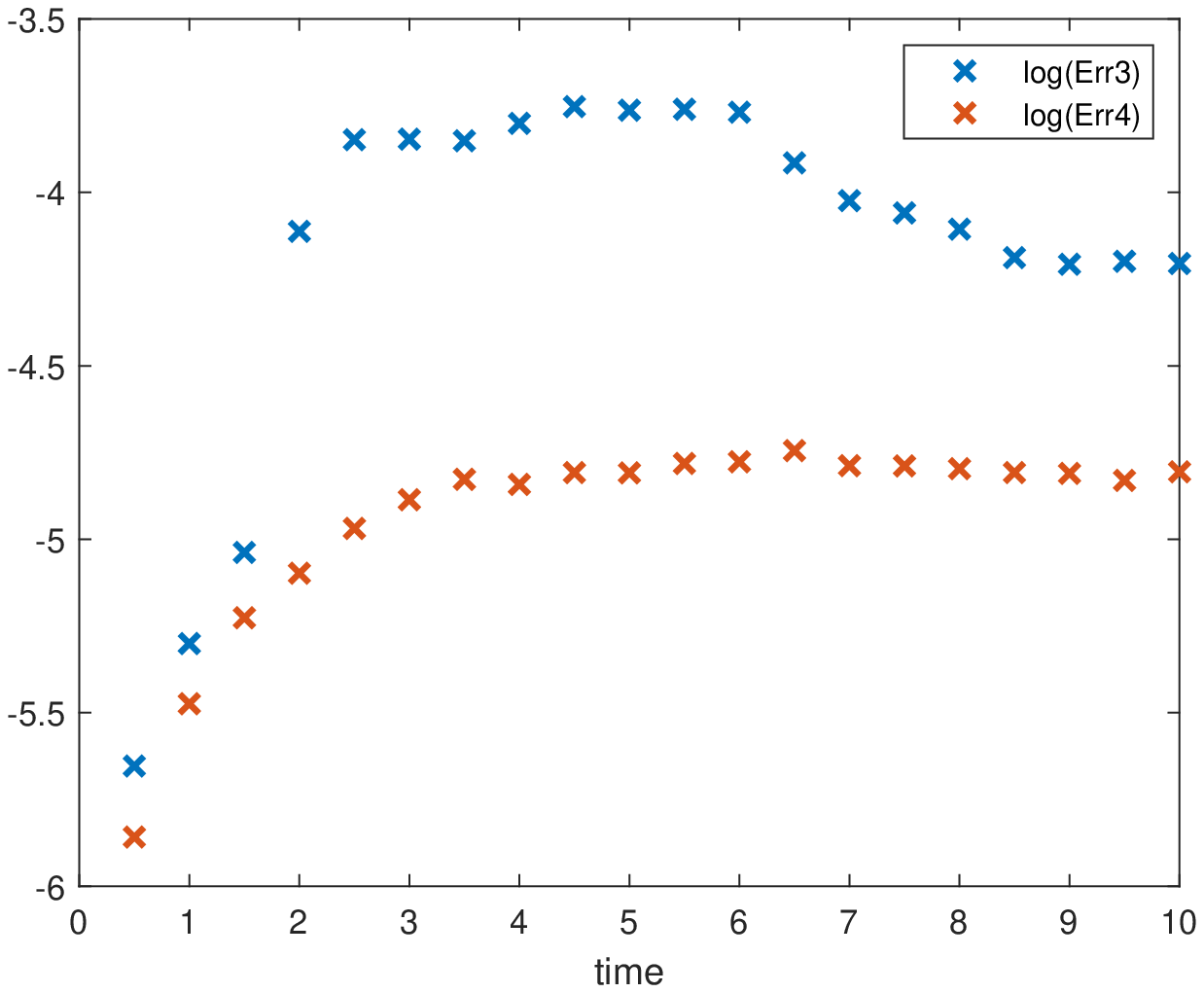}}
\qquad
\caption{Logarithm of $Err_1$ and $Err_2$ in (a) and of $Err_3$ and $Err_4$ in (b) in time for $N=3200$ and $\Delta x = \frac{1}{160}$.}
\label{fig10}
\end{figure}

Varying the step sizes of the space grid in the macroscopic model we obtain the following errors and empirical convergence rates in Table 2 and, respectively. All error measures decrease for decreasing step sizes whereas the convergence rates tend to increase for decreasing step sizes.

\begin{table}[htb!]
\centering

\begin{minipage}{0.47\textwidth}
\begin{tabular}{c|c|c|c|c}
\hline
  &   $Err_1$ & $Err_2$ & $Err_3$ & $Err_4$ \\
\hline
$\Delta x=\frac{1}{40}$ & $0.1112$ & $0.0440$  & $0.6619$ & $0.0483$ \\
\hline
$\Delta x=\frac{1}{80}$ & $0.0678$ & $0.0371$  & $0.2546$ & $0.0479$ \\
\hline
$\Delta x=\frac{1}{160}$ & $0.0453$ & $0.0320$  & $0.1040$ & $0.0358$ \\
\hline
\end{tabular}
\label{tab:2}
\caption{Error measures for different step sizes $\Delta x$ and $N=3200$.}
\end{minipage} \hfill
\begin{minipage}{0.47\textwidth}
\begin{tabular}{c|c|c|c|c}
\hline
  &   $Err_1$ & $Err_2$ & $Err_3$ & $Err_4$ \\
\hline
$\Delta x=\frac{1}{40}$ & $0.661$ & $0.778$  & $0.298$ & $0.807$ \\
\hline
$\Delta x=\frac{1}{80}$ & $0.802$ & $0.821$  & $0.580$ & $0.819$ \\
\hline
$\Delta x=\frac{1}{160}$ & $0.840$ & $0.846$  & $0.653$ & $0.884$ \\
\hline
\end{tabular}
\label{tab:3}
\caption{Empirical convergence rates \\for different step sizes $\Delta x$ and $N=3200$.}
\end{minipage}
\end{table}

\section{Conclusion}
We introduced a microscopic and a macroscopic traffic accident model in which accidents interact bi-directional with the traffic situation. Accidents were incorporated using appropriate probability measures. We were able to prove a micro-macro limit of the models, restricting to a microscopic model being governed by the macroscopic one. The numerical simulations underlined this convergence and indicated that it can be extended for the microscopic model from Section $\ref{sec:BasicModel}$. 

The analytic proof of the convergence of the two microscopic models might be subject to future work. Additionally, a data driven validation of the models and a deeper investigation of the accident occurrences will be addressed in future. 

\section*{Acknowledgement}
This work was supported by the DAAD project "Stochastic dynamics for complex networks and systems" (Project-ID 5744394).

\printbibliography

\end{document}